\documentclass[11pt]{article}

	\usepackage[utf8]{inputenc}
	\usepackage[T1]{fontenc}
	\usepackage[english]{babel}
	\usepackage{amsmath, amsthm, amsfonts, amssymb}
	\usepackage{bm, enumerate, graphicx, array, xcolor}
	\usepackage{cite, hyperref}
	\usepackage{tikz}
	\usetikzlibrary{trees}
	
	\newcolumntype{P}[1]{>{\centering\arraybackslash}p{#1}}
	\newcolumntype{M}[1]{>{\centering\arraybackslash}m{#1}}

\theoremstyle{plain}
	\newtheorem{lemma}{Lemma}
	\newtheorem{proposition}{Proposition}
	\newtheorem{theorem}{Theorem}
	
\theoremstyle{definition}
	\newtheorem{definition}{Definition}
\theoremstyle{remark}
	\newtheorem{remark}{Remark}
	\newtheorem{example}{Example}
    \newtheorem{algorithm}{Algorithm}	
    
\newcommand{\fn}{\mathfrak{n}}
\newcommand{\fg}{\mathfrak{g}}

\newcommand{\fu}{\mathfrak{u}}
\newcommand{\fm}{\mathfrak{m}}
\newcommand{\fr}{\mathfrak{r}}
\newcommand{\fs}{\mathfrak{s}}
\newcommand{\fgl}{\mathfrak{gl}}
\newcommand{\fsl}{\mathfrak{sl}}

\newcommand{\ku}{\mathbb{K}}
\newcommand\blank{{\mkern 2mu\cdot\mkern 2mu}}

\DeclareMathOperator{\ad}{ad} 
\DeclareMathOperator{\der}{Der} 
\DeclareMathOperator{\id}{Id} 
\DeclareMathOperator{\spa}{span}

\title{Lie structures and chain ideal lattices}
\author{Pilar Benito and Jorge Roldán-López}
\date{December 2, 2022}

\begin{document}
\maketitle

\begin{abstract}
    The purpose of this paper is twofold. Firstly, to emphasise that the class of Lie algebras with chain lattices of ideals are elementary blocks in the embedding or decomposition of Lie algebras with finite lattice of ideals. Secondly, to show that the number of Lie algebras of this class is large and they support other types of Lie structures. Beginning with general examples and algebraic decompositions, we focus on computational algorithms to build Lie algebras in which the lattice of ideals is a chain. The chain condition forces gradings on the nilradicals of this class of algebras. Our algorithms yield to several positive naturally graded parametric families of Lie algebras. Further generalizations and other kind of structures will also be discussed.
\end{abstract}



\section{Introduction}


The algebraic structure of a Lie algebra imposes strong properties on its lattice of ideals. Although, the lattice of ideals does not always determine the Lie algebra in a unique way, many fundamental properties of Lie algebras can be interpreted as properties about their lattices. The lattice of ideals of a \emph{thin algebra} is a sequence of diamonds (subspaces of a two-dimensional vector space) connected by chains (see \cite[Section 1]{mattarei2022constituents}). For finite-dimensional algebras, according to \cite[Corollary 2.10]{Benito_2020}, reductive Lie algebras (i.e., direct sum of ideals of semisimple by abelian) are just the class of Lie algebras with complemented lattice of ideals. If we impose the additional condition of finiteness to a complemented lattice, we get the class of semisimple or semisimple by one-dimensional Lie algebras, whose lattices of ideals are hypercubes. In fact, finite lattices of ideals of Lie algebras are sublattices of hypercube lattices (see~\cite[Theorem 2.9]{Benito_2020}). 

From Dilworth's Chain Decomposition Theorem \cite{dilworth2009decomposition}, any finite lattice decomposes as a disjoint union of chains. This fact highlights $n$-element chain lattices as basic blocks for embedding or decomposing finite lattices. In \cite[Theorem 2.2 and Theorem 3.4]{pilar1992lie}, structural-theoretic characterizations of finite-dimensional Lie algebras of characteristic zero whose lattice of ideals is a chain are given. Over algebraically closed fields, the complete list of solvable Lie algebras in this family appears in Theorem 3.4 (basis and bracket description). For dimension greater or equal than $2$, these algebras are one-dimensional extensions of some nilpotent Lie algebra $\fn$ by an invertible self-derivation $d$ with integer positive eigenvalues. So $\fg=\ku\cdot d\oplus \fn$ and $\fn$ is either a \emph{generalised Heissenberg algebra}, or a \emph{filiform} one if the nilpotent index is greater or equal than $3$ or a \emph{thin algebra} with two diamonds (the centre and $\fn/\fn^2$ have dimension two and $\frac{\fn^i}{\fn^{i+1}}$ is one-dimensional otherwise). The common pattern to all $\fn$ is that they are \emph{naturally graded} algebras and generated (as algebras) by a subspace $\fu$, so $\fn=\oplus_{k\geq 1}\fu^k$ with $\fu^1=\fu$ and $\fu^{i}=[\fu,\fu^{i-1}]$. A Lie algebra satisfying previous conditions is called \emph{homogeneous} or \emph{quasi-cyclic} (first definition in \cite{leger1963derivations}). Even more, each nilradical $\fn$ is a flexible Lie algebra and the whole algebra $\fg$ is semi-contractable ($\fg_0=\ku\cdot d$) according to \cite[Definition 3.1]{cornulier2016gradings}.

There is a large number of different types of mixed Lie algebras with chain lattice of ideals, even over algebraically closed fields. The easier example is given by any split extension of a simple Lie algebra and a nontrivial irreducible module (abelian nilradical and $3$-chain). In characteristic zero, the well-known Levi Theorem lets us split a mixed chained-lattice Lie algebra as $\fs\oplus_\rho \fn$ where $\fs$ is a simple subalgebra and their radical $\fn$ is nilpotent. The chain ideal condition forces strong patterns on the faithful representation $\rho\colon \fs\to \der \fn$ and positively gradings on the nilradical are found.

In this paper we develop computer algorithms which let us build mixed Lie algebras whose ideals form a $n$-chain for $n\leq 6$. These algorithms will produce parametric families of quasi-cyclic or Carnot graded Lie algebras of arbitrary dimension and nilpotent index up to $5$. The algebras are easily described by taking basis and defining their respective structure constants. These algebras are quotient of free nilpotent algebras by homogeneous ideals and admit expanding and partially expanding automorphisms (check~\cite{dere2017gradings}). Starting with the split $3$-dimensional Lie algebra $\mathfrak{sl}_2(\mathbb{K})$ and using its irreducible representations $V_n$ and some suitable $\mathfrak{sl}_2(\mathbb{K})$-invariant bilinear products $V_n\otimes V_m\to V_{m+n-2k}$ that appeared in \cite{Dixmier_1984}, we propose a block construction of Lie algebras with chain ideal lattices. These algorithms are inspired by \cite{Dixmier_1984} and \cite{bremner2004invariant}, and their specifications are presented at the end of Subsection~\ref{sub:algorithm}. Previously, Section~\ref{s:generalities} and Section~\ref{s:algorithm} exhibit many more examples and introduce generalities and tools that are the theoretical environment of the algorithms. In Section~\ref{s:lemmas}, we have included several results on existence induced by applying the mentioned algorithms. The existence results produce nilpotent positively $\mathbb{Z}$-graded Lie algebras that support additional structures as it is explain in the final comments in Section~\ref{FC} (invariant metrics, expanding automorphisms and left-symmetric structures among others).

Basic results on Lie algebras follow from \cite{Jacobson_1979} and \cite{Humphreys_1997}. Along this paper, all vector spaces are of finite dimension over a field~$\ku$ of characteristic zero unless otherwise stated.

\section{Generalities and examples}\label{s:generalities}

A Lie algebra $\fg$ is a vector space over a field $\mathbb{K}$ endowed with a binary skew-symmetric ($\frac{1}{2}\in \mathbb{K}$) bilinear product $[x,y]$ satisfying the Jacobi identity:
\begin{equation}\label{eq:Jacobi}
J(x,y,z)=[[x,y],z]+[[z,x],y]+[[y,z],x]=0\quad \forall\, x,y,z\in \fg.
\end{equation}
In case $[x,y]=0$ for every $x,y\in \fg$, the Lie algebra $\fg$ is called \emph{abelian} and identity~\eqref{eq:Jacobi} becomes trivial. Along this paper, we will denote as the Lie bracket of two vector subspaces $U,V$ of $L$ to the whole linear span
\begin{equation*}
[U,V]=\spa\langle [u,v]: u\in U, v\in V\rangle.
\end{equation*}
A subspace $U$ of $\fg$ is subalgebra (or ideal) of $\fg$ if $[U,U] \subseteq U$ (or $[\fg,U] \subseteq U$). The derived series of $\fg$ is defined recursively as $\fg^{(1)}=\fg$ and $\fg^{(n)}=[\fg^{(n-1)},\fg^{(n-1)}]$ for $n>1$. While the lower central series (LCS) is defined as $\fg^1=\fg$ and $\fg^n=[\fg,\fg^{n-1}]$ for $n>1$. The terms $\fg^i$ and $\fg^{(i)}$ are ideals of $\fg$. If the derived series (or LCS) vanishes, $\fg$ is called \emph{solvable} (or \emph{nilpotent}). The \emph{solvable radical} (or \emph{nilpotent radical}) of $\fg$, denoted as $\fr(\fg)$ (or $\fn(\fg)$) is the biggest solvable (or nilpotent) ideal of $\fg$. We will also denote these ideals as $\fr$ (or $\fn$). The nilindex or index of nilpotency of a nilpotent Lie algebra is the smallest integer such $t$ such that $\fn^t=0$. Along this paper we refer to nilpotent Lie algebras of nilindex $t$ as $t$-nilpotent algebras or $(t-1)$-step nilpotent. And the  $(t-1)$-tuple $(c_1,\dots, c_{t-1})$ in which the $i^{th}$ component is $c_i=\dim n^{i}/n^{i+1}$ will be called the \emph{general type  of $\fn$} and $c_1$ will also be termed \emph{type of $\fn$}.

A \emph{semisimple} Lie algebra is by definition an algebra with no non-zero solvable ideals, and it decomposes as a direct sum of ideals which are simple Lie algebras. A \emph{simple} Lie algebra is a non-abelian Lie algebra that contains no nonzero proper ideals. Levi's Theorem asserts that a Lie algebra $\fg$ decomposes as a direct sum $\fg=\fs\oplus \fr$ where $\fs$ is a semisimple subalgebra (Levi subalgebra). In this paper, a \emph{mixed Lie algebra} is a non solvable and non semisimple algebra.

A derivation of $\fg$ is a linear map $d\colon\fg\to \fg$ such that $d[x, y] = [d(x), y]+[x, d(y)]$. $\der \fg$, the whole set of derivations of $\fg$, is a Lie subalgebra of the general Lie algebra $\fgl(\fg)$, the linear maps of $\fg$ under the commutator bracket $[f,g]=fg-gf$. For any $x\in \fg$, the map $\ad x (y)=[x,y]$ is a derivation called \emph{inner derivation}. The general Lie algebra, $\fgl(V)$ (or $\fgl(V,\ku)$ or $\fgl_n(\ku)$ for matrix description) can be defined for any $\ku$-vector space $V$ and the set of traceless linear maps, $\fsl(V)$ is a Levi subalgebra.

For a vector space $V$, a representation of a Lie algebra $\fg$ is an homomorphism of Lie algebras $\rho\colon \fg\to\mathfrak{gl}(V)$. The vector space $V$ under the action $x\,\cdot\,v=\rho(x)(v)$ is called $\fg$-module. By taking $\rho=\id$, the vector $V$ is a \emph{natural module} for any subalgebra $\fs$ of $\fgl(V)$, so $\rho(f)(v)=f(v)$. The adjoint representation $\rho=\mathbb{\ad}$ of $\fg$ is $V=\fg$ and $\rho(x)(y)=[x,y]$.  A module $V$ is irreducible if it is nontrivial and does not contain proper submodules. In case the kernel of $\rho$ is trivial, $V$ is said to be a \emph{faithful} $\fg$-module. Any representation of any semisimple Lie algebra is completely reducible. New modules can be obtained from old ones: as a quotient of a module by a submodule in the natural way or as a tensor product $V\otimes W$ (or $\Lambda^n V$ and $S^nV$) of modules by declaring
\begin{equation*}
	x\cdot(v\otimes w)=(x\cdot v)\otimes w+v\otimes(x\cdot w).
\end{equation*}

A \emph{naturally graded} (respectively positive naturally graded) algebra or $\mathbb{N}$-graded (respectively $\mathbb{N}^+$-graded) algebra, is a Lie algebra $\fg$ that admits a nontrivial grading on the set of natural numbers $\mathbb{N}$ including $0$ (respectively excluding $0$).
\begin{definition}
    A finite-dimensional quasi-cyclic (homogeneous or Carnot) Lie algebra is a positive naturally graded algebra $\fg=\fg_1\oplus\fg_2\oplus \dots \oplus \fg_t$ generated as an algebra by $\fg_1$. This means $[\fg_i,\fg_j]\subseteq \fg_{i+j}$ (here $\fg_{s}=0$ for $s>t$) and $\fg_i=[\fg_1,\fg_1^{i-1}]$.
\end{definition}
\noindent Following \cite[Definition 3.3]{cornulier2016gradings} the terms quasi-cyclic, graded or homogeneous are better known in Lie algebras, while the word Carnot (graded) is more commonly used in sub-Riemannian and conformal geometry.

The set of ideals of a Lie algebra $L$ form a lattice as they are a poset (partially ordered set) ordered by inclusion, and given two elements we have an infimum (intersection) and supremum (sum). Lattices of ideals of Lie algebras are modular, and distributive if we impose finiteness.  A finite lattice can be represented through its Hasse diagram, which is a graph where nodes are the ideals and there is a link between two ideals when they are contained and there is no other ideal between. A modular lattice (or distributive) does not contain pentagons (or diamonds) as sublattices.

\begin{figure}[!ht]
	\centering
		\begin{tikzpicture}
			\node (Z1) at (0,0) {};
			\node (L1) at (-0.8, 1) {};
			\node (L1T) at (-1.07,1) {$c$};
			\node (R1) at (0.8, 0.66) {};
			\node (R1T) at (1.07,0.66) {$b$};
			\node (R2) at (0.8, 1.33) {};
			\node (R2T) at (1.07,1.33) {$a$};
			\node (H1) at (0,2) {};
	
			\filldraw [black] (Z1) circle (1.5pt);
			\filldraw [black] (L1) circle (1.5pt);
			\filldraw [black] (R1) circle (1.5pt);
			\filldraw [black] (R2) circle (1.5pt);
			\filldraw [black] (H1) circle (1.5pt);
			
			\draw (Z1) -- (L1) -- (H1);
			\draw (Z1) -- (R1) -- (R2) --(H1);
		\end{tikzpicture}	
		\hspace{1cm}
		\begin{tikzpicture}
			\node (Z2) at (0,0) {};
			\node (C1) at (-0.8, 1) {};
			\node (C2) at (0, 1) {};
			\node (C3) at (0.8, 1) {};
			\node (C1T) at (-1.07, 1) {$a$};
			\node (C2T) at (-0.27, 1) {$b$};
			\node (C3T) at (1.07, 1) {$c$};
			\node (H2) at (0,2) {};
	
			\filldraw [black] (Z2) circle (1.5pt);
			\filldraw [black] (C1) circle (1.5pt);
			\filldraw [black] (C2) circle (1.5pt);
			\filldraw [black] (C3) circle (1.5pt);
			\filldraw [black] (H2) circle (1.5pt);
			
			\draw (Z2) -- (C1) -- (H2);
			\draw (Z2) -- (C2) -- (H2);
			\draw (Z2) -- (C3) -- (H2);
		\end{tikzpicture}
	\caption{Pentagon and diamond lattices.}
	\label{fig:nonlattices}
\end{figure}
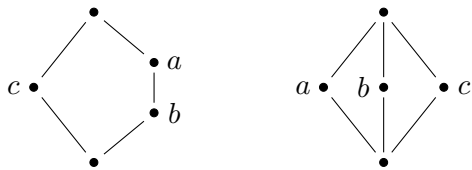
In Figure~\ref{fig:nonlattices} we have which sublattices must not appear in case we want a modular and distributive lattice, whereas in Figure~\ref{fig:latticesExamples} we can find some examples of lattices of ideals of Lie algebras.

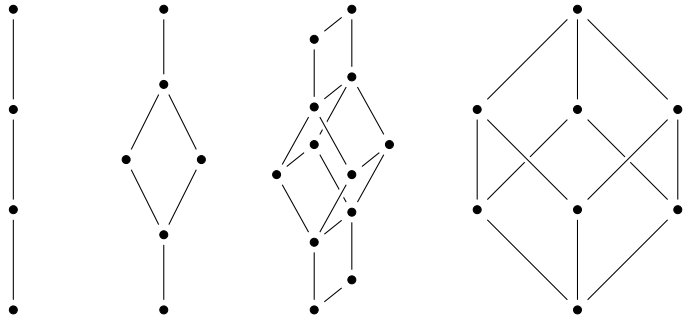
\begin{figure}[!ht]
	\centering
	\newcommand{\hd}{4}
	\newcommand{\drp}{0}
	\newcommand{\dcp}{2}
	\newcommand{\dhp}{4}
	\newcommand{\dbp}{7.5}
	\newcommand{\ax}{0.5}
	\newcommand{\ay}{0.4}
	\newcommand{\hdn}{3.6}
	\newcommand{\anchCubo}{4/3}
	\newcommand{\altoCubo}{4/3}
	\newcommand{\sepYCubo}{4/3}
	\newcommand{\psize}{1.5pt}
	\begin{tikzpicture}
			\node (DR4)  at (\drp,    \hd) {};
			\node (DR3)  at (\drp,    \hd*2/3) {};
			\node (DR2)  at (\drp,    \hd*1/3) {};
			\node (DR1)  at (\drp,    0) {};
			
			\filldraw [black] (DR4) circle (\psize);
			\filldraw [black] (DR3) circle (\psize);
			\filldraw [black] (DR2) circle (\psize);
			\filldraw [black] (DR1) circle (\psize);
			
			\draw (DR1) -- (DR2) -- (DR3) -- (DR4);
		
			\node (DC5)  at (\dcp,    \hd) {};
			\node (DC4)  at (\dcp,    \hd*3/4) {};
			\node (DC32) at (\dcp-0.5,\hd*2/4) {};
			\node (DC31) at (\dcp+0.5,\hd*2/4) {};
			\node (DC2)  at (\dcp,    \hd*1/4) {};
			\node (DC1)  at (\dcp,    0) {};
			
			\filldraw [black] (DC5)  circle (\psize);
			\filldraw [black] (DC4)  circle (\psize);
			\filldraw [black] (DC32) circle (\psize);
			\filldraw [black] (DC31) circle (\psize);
			\filldraw [black] (DC2)  circle (\psize);
			\filldraw [black] (DC1)  circle (\psize);
			
			\draw (DC1) -- (DC2) -- (DC31) -- (DC4) -- (DC5);
			\draw (DC2) -- (DC32) -- (DC4);
		
			\node (DH5)  at (\dhp,    \hdn) {};
			\node (DH4)  at (\dhp,    \hdn*3/4) {};
			\node (DH32) at (\dhp-0.5,\hdn*2/4) {};
			\node (DH31) at (\dhp+0.5,\hdn*2/4) {};
			\node (DH2)  at (\dhp,    \hdn*1/4) {};
			\node (DH1)  at (\dhp,    0) {};
			
			\node (D2H5)  at (\dhp+\ax,    \hdn+\ay) {};
			\node (D2H4)  at (\dhp+\ax,    \hdn*3/4+\ay) {};
			\node (D2H32) at (\dhp-0.5+\ax,\hdn*2/4+\ay) {};
			\node (D2H31) at (\dhp+0.5+\ax,\hdn*2/4+\ay) {};
			\node (D2H2)  at (\dhp+\ax,    \hdn*1/4+\ay) {};
			\node (D2H1)  at (\dhp+\ax,    \ay) {};
			
			\filldraw [black] (DH5)  circle (\psize);
			\filldraw [black] (DH4)  circle (\psize);
			\filldraw [black] (DH32) circle (\psize);
			\filldraw [black] (DH31) circle (\psize);
			\filldraw [black] (DH2)  circle (\psize);
			\filldraw [black] (DH1)  circle (\psize);
			\filldraw [black] (D2H5)  circle (\psize);
			\filldraw [black] (D2H4)  circle (\psize);
			\filldraw [black] (D2H32) circle (\psize);
			\filldraw [black] (D2H31) circle (\psize);
			\filldraw [black] (D2H2)  circle (\psize);
			\filldraw [black] (D2H1)  circle (\psize);
			
			\draw (D2H1) -- (D2H2) -- (D2H31) -- (D2H4) -- (D2H5);
			\draw (D2H2) -- (D2H32) -- (D2H4);
			\draw [white,line width=1mm] (DH2)--(DH31)--(DH4);
			\draw (DH1) -- (DH2) -- (DH31) -- (DH4) -- (DH5);
			\draw (DH2) -- (DH32) -- (DH4);

			\draw (DH5)  -- (D2H5);
			\draw (DH4)  -- (D2H4);
			\draw (DH32) -- (D2H32);
			\draw (DH31) -- (D2H31);
			\draw (DH2)  -- (D2H2);
			\draw (DH1)  -- (D2H1);
			
		
			\node (Z1) at (          \dbp,         0) {};
			\node (I1) at (\dbp-\anchCubo, \altoCubo) {};
			\node (D1) at (\dbp+\anchCubo, \altoCubo) {};
			\node (L1) at (          \dbp, \altoCubo*2) {};
			
			\node (Z2) at (          \dbp,\sepYCubo) {};
			\node (I2) at (\dbp-\anchCubo,\sepYCubo+\altoCubo) {};
			\node (D2) at (\dbp+\anchCubo,\sepYCubo+\altoCubo) {};
			\node (L2) at (          \dbp,\sepYCubo+\altoCubo*2) {};
	
			\filldraw [black] (Z1) circle (\psize);
			\filldraw [black] (Z2) circle (\psize);
			\filldraw [black] (I1) circle (\psize);
			\filldraw [black] (I2) circle (\psize);
			\filldraw [black] (D1) circle (\psize);
			\filldraw [black] (D2) circle (\psize);
			\filldraw [black] (L1) circle (\psize);
			\filldraw [black] (L2) circle (\psize);
			
			\draw (Z1) -- (I1) -- (L1);
			\draw (Z1) -- (D1) -- (L1);
			\draw [white,line width=1mm] (D2)--(Z2)--(I2);
			\draw (Z2) -- (I2) -- (L2);
			\draw (Z2) -- (D2) -- (L2);
			\draw (Z1) -- (Z2);
			\draw (D1) -- (D2);
			\draw (I1) -- (I2);
			\draw (L1) -- (L2);
	\end{tikzpicture}
	
	\caption{Examples of lattices of Lie algebras.}
	\label{fig:latticesExamples}
\end{figure}

Note, in Figure~\ref{fig:latticesExamples} we can find a chain of ideals which is the  lattice in which we are focus through this paper. It is worth mentioning that a Lie algebra determines a unique lattice, but the contrary does not hold. This is why this chain lattice in Figure~\ref{fig:latticesExamples} is associated to different mixed algebras, or even to the solvable oscillator Lie algebra in the real field. This is a 4-dimensional Lie algebra which can be seen as
	\begin{equation*}
		\mathfrak{d}_4 = \spa_\mathbb{K}\langle d \rangle \oplus \mathfrak{h}_3
	\end{equation*}
	where $\mathfrak{h}_3$ is the Heisenberg algebra with basis $\{x,y,z\}$. Thus, we have
	\begin{alignat*}{3}
		[x,y] &= z,\qquad\quad\enspace  & [x,z] &= 0,\qquad\qquad  & [y,z]&=0,\\
		[d,x] &= y,& [d,y] &= -x, &[d,z]&=0.
	\end{alignat*}
	Note, this same oscillator algebra, in the complex field, produces the second lattice of ideals in Figure~\ref{fig:latticesExamples}. Also, we can see how the third lattice comes from a sort of duplication of the second one. This duplication can be achieved, for example, by a direct sum of a simple Lie algebra as a trivial extension. Indeed, this also explains why $\mathfrak{sl}_2(\ku) \oplus \mathfrak{sl}_2(\ku) \oplus \mathfrak{sl}_2(\ku)$ has by its lattice the fourth one in the same figure, as it comes by duplicating two times the chain of two ideals. We also note that the direct sum as ideals of a Lie algebra with chain lattice by a simple one produces a \emph{ladder lattice}. The deconstruction of the lattices in Figure~\ref{fig:latticesExamples} into chains stated by Dilword's Theorem is clear.
	
	In the sequel we will show that chained-lattice mixed Lie algebras are a very large family.

According to \cite{gauger_1973}, any nilpotent Lie algebra $\fn$ of $(t+1)$-nilindex and minimal generator vector space $\fm$ of dimension $d$ (type $d$) is a quotient of the free nilpotent algebra $\fn_{d,t}$ on a set of $d$ generators by an ideal $I$ such that $\fn_{d,t}^t\nsubseteq I\subseteq \fn_{d,t}^2$. Up to isomorphisms, the Levi subalgebra of the derivation Lie algebra $\der \fn_{d,t}$ is $\fsl_d(\ku)$. By using multilinear algebra, models of free nilpotent algebras of nilindex $3$ and $4$ are easily obtained:
\begin{alignat*}{2}
    \fn_{d,2}&=\fm \oplus \Lambda^2 \fm,& [u,v]&=u\wedge v,\\
    \fn_{d,3}&=\fm \oplus \Lambda^2\fm\oplus \frac{\fm \otimes \Lambda^2\fm}{\Lambda^3\fm},\qquad & [u,v\wedge w]&=u\otimes v\wedge w\ \mod \Lambda^3\fm.
\end{alignat*}
Any linear map $f\colon \fm \to \fm$ extends to a derivation $d_f$ of $\fn_{d,t}$. The set of such extension maps that have zero trace is just the Levi subalgebra of $\der \fn_{d,t}$ that we also denote as $\mathfrak{sl}_d(\fm, \ku)$ (for short $\fsl_d(\ku)$ or $\fsl(\fm)$). As $\mathfrak{sl}(\fm)$-modules, for $d\geq 3$, $\fm=V(\lambda_1)$,  $\Lambda^2\fm=V(\lambda_2)$ and $\Lambda^3\fm=\ku$ if $d=3$ and $\Lambda^3\fm=V(\lambda_3)$ if $d\geq 4$. In addition,
\begin{equation*}
\frac{\fm \otimes \Lambda^2\fm}{\Lambda^3\fm}=V(\lambda_1+\lambda_2).
\end{equation*}
Denoting by $\rho_1$ and $\rho_2$ the natural representations of $\mathfrak{sl}(\fm)$ on $\fn_{d,2}$ and $\fn_{d,3}$, we arrive at the series of mixed Lie algebras with $4$-chain and $5$-chain ideals:
\begin{equation}\label{eq:n2t-4-5-cadena}
    \mathfrak{sl}(\fm)\oplus_{\rho_1} \fn_{d,2}\quad \text{and}\quad \mathfrak{sl}(\fm)\oplus_{\rho_2} \fn_{d,3}.
\end{equation}
\begin{example}\label{sl-low}
The smallest Lie algebras in equation~\eqref{eq:n2t-4-5-cadena} correspond to a vector space $\fm$ of dimension $2$. In this case, we get algebras of dimension $6$ and $8$ with $\mathfrak{sl}_2(\ku)$-irreducible decomposition $V_2\oplus V_1\oplus V_0$ and $V_2\oplus V_1\oplus V_0\oplus V_1$. Here $V_n$ is the $(n+1)$-dimensional irreducible module of $\mathfrak{sl}_2(\ku)$ and  $V_2$ is just the adjoint module of  $\mathfrak{sl}_2(\ku)$ (see next section for a complete description of the algebras using basis and bracket product).
\end{example}

\begin{example}\label{so-chain}
From the set of skew-maps of a fixed vector space $\fm$ relative to a bilinear and non-degenerate form $\varphi$, either symmetric or skew-symmetric, we get classical simple Lie algebras $\mathfrak{so}(\fm, \varphi)$ (types $B$ or $D$ if $\varphi$ is symmetric and depending on whether $\dim \fm$ is odd or even) and $\mathfrak{sp}(\fm, \varphi)$ (type $C$ if $\varphi$ is skew and $\dim \fm$ is even). The natural module $\fm=V(\lambda_1)$ of $\mathfrak{so}(\fm, \varphi)$ ($\dim \fm \geq 7$ is required for tensor decompositions) provides  the representation $\rho_i$, (it is the restricted representation of the one given in equation~\eqref{eq:n2t-4-5-cadena}). Then, $\Lambda^2\fm=V(\lambda_2)$ and
\begin{equation*}
   Z(\fn_{d,3})=\frac{\fm \otimes \Lambda^2\fm}{\Lambda^3\fm}=V(\lambda_1+\lambda_2)\oplus V(\lambda_1), 
\end{equation*}
and we get the series of Lie algebras of seven ideals
    \begin{equation*}
    \mathcal{L}(\fm)=\mathfrak{so}(\fm, \varphi)\oplus_{\rho_2} \fn_{d,3}.
    \end{equation*}
The lattice of ideals of $\mathcal{L}(\fm)$ is a 4-element chain connected by the rhombus ideal $Z(\fn_{d,3})$ at the bottom. But the quotient Lie algebras by minimal ideals inside $Z(\fn_{d,3})$,
\begin{equation*}
    \frac{\mathcal{L}(\fm)}{V(\lambda_1)}\quad \text{and}\quad  \frac{\mathcal{L}(\fm)}{V(\lambda_1+\lambda_2)},
\end{equation*}
are $5$-chain mixed Lie algebras. In both cases we get the $5$-chain by removing a minimal node in the complete lattice of ideals of $\mathcal{L}(\fm)$.
\end{example}

\begin{example}
For any $n\geq 1$, $\mathfrak{h}_n$ denotes the $n^{th}$ generalised Heisenberg Lie algebra of dimension $2n+1$. This algebra has a \emph{standard basis} $e_1,\dots, e_n,$ $e_{n+1}, \dots e_{2n}, z$ with nonzero brackets $[e_i,e_{n+i}]=z$. The vector space $\fm=\spa\langle e_1,\dots, e_n,e_{n+1}, \dots e_{2n}\rangle$ is endowed with the non-degenerate skew-form $[a,b]=\varphi(a,b)z$. According to \cite[Example~1]{benito2013levi}), the Levi factor of $\der \mathfrak{h}_n$ is the Lie algebra $\fs$ of extended maps $d_f$ where and $d_f(z)=0$ and $d_f|_{\fm}=f$ for every $f\in \mathfrak{sp}(\fm, \varphi)$. So $\fs \cong \mathfrak{sp}(\fm, \varphi)$ and $\mathfrak{h}_n$ decomposes as the natural $\mathfrak{sp}(\fm, \varphi)$-module $\fm=V(\lambda_1)$, and the trivial one-dimensional module $Z(\mathfrak{h}_n)=\ku\cdot z$. Clearly, the ideals of the Lie algebra $\fs\oplus_{id}\mathfrak{h}_n$ (a mixed subalgebra of the mixed algebra $\der \mathfrak{h}_n\oplus_{id}\mathfrak{h}_n$) form a 4-chain. For $n=1$, since $\mathfrak{sp}_2(\ku)\cong \fsl_2(\ku)$, previous $4$-chain Lie algebra is encoded in a Lie structure $V_2\oplus V_1\oplus V_0$ in Example~\ref{sl-low}. Here $V_1\oplus V_0$ is just $\mathfrak{h}_1$. According to \cite[Example~6.6]{kharraf2021classification}, for any $n\geq 1$, the Heisenberg algebra $\mathfrak{h}_n$ can be endowed with a structure of a simple \emph{Hom-Lie algebra}. For the whole description of Hom-Lie structures on $\mathfrak{h}_1$ see \cite{alvarez2019cohomology}.
\end{example}

\begin{example}
The tensor product $S\otimes A$ of a Lie algebra $S$ by a commutative and associative algebra $A$ produces a Lie algebra named in the literature \emph{current Lie algebra of $S$ by $A$}. \emph{Poisson structures} and invariant bilinear forms on current Lie algebras are treated in \cite[Theorem~2,  Corrollary~2.2 and Lemma~2.3]{zusmanovich2014compendium}. If $A$ has unit, a copy of $S$ appears as subalgebra of $S\otimes A$. If $S$ is simple, the ideals of $S\otimes A$ are of the form $S\otimes I$ where $I$ is an ideal of $A$. Since the Killing form is an invariant and non-degenerate form of any simple Lie algebra, from Lemma 2.3 in \cite{zusmanovich2014compendium}, the current Lie algebra $S\otimes A$ can also be endowed with an invariant symmetric and nondegenerate bilinear form. In this way we get a \emph{metric Lie structure} (\emph{quadratic or metrizable Lie algebra}). Consider now the series of current algebras 
\begin{equation*}
    \fg_n(S)=S \otimes \frac{\ku[t]}{\spa\langle t^n\rangle}
\end{equation*}
for $S$ a simple Lie algebra. Note, the block $\fs=S\otimes 1$ is a Levi subalgebra of $\fg_n$, and the solvable radical, $\fr(\fg_n)=\fn(\fg_n)=\oplus_{i=1}^{n-1}S\otimes x^i$, is a \emph{positive naturally graded} Lie algebra (here $x$ is the class of the element $t^i$ mod $\spa\langle t^n\rangle$) generated by $S\otimes x$. So, $\fn(\fg_n)$ is Carnot and the whole algebra $\fg_n$ is also naturally graded. As $\fs$-module, $\fg_n$ decomposes as the direct sum of $n$ copies of the adjoint module of $S$ and its lattice of ideals is a $(n+1)$-element chain. In addition, $\fg_n$ is a \emph{quadratic Lie algebra}. The smallest algebras appear by taking $S=\fsl_2(\ku)$. The $\fsl_2(\ku)$-module decomposition of $\fg_n(\fsl_2(\ku))$ is $V_2\oplus\dots \oplus V_2$ ($n$ summands).
\end{example}

\section{Theoretical support, tools and algorithms}\label{s:algorithm}

The anticonmutivity and Jacobi identity are the identities that determine any Lie algebra $\fg$. The first one is equivalent to say that the product $[x,y]$ on $\fg$ in is given by a bilinear map $\Lambda^2 \fg\to \fg$; while the latest is equivalent to state that the right multiplication $\ad x$ is a derivation of $\fg$, for every $x\in \fg$. If $\fg$ is simple, it is irreducible as adjoint module and a copy of $\fg$ is inside $\Lambda^2 \fg$. Reversing and generalising this argument, for an irreducible representation $\rho$ of a semisimple Lie algebra $\fs$ over a vector space $V$, the existence of a copy of $V$ inside $\Lambda^2V$ let us define an skew-product $\star\colon V\otimes V \to V$ such that $\rho(\fs)\subseteq \der (V, *)$. This induces naturally a Lie structure on the vector space $\fs\oplus_\rho V$. Along this section, we follow this idea in order to get Lie algebras with chain ideal lattices.

Our algorithms to give the desired Lie structure are based on the representation theory of $\fsl_2(\mathbb{K})$ and the use of transvections to express skew-products, and the structure results given in~\cite[Theorem~2.2]{pilar1992lie} and \cite[Theorem~2]{Snobl_2010}. Both theorems can be found below.
\begin{theorem}[Benito, 1992]\label{thm:basiconPi}
Let $\fg$ be a mixed Lie algebra. Then, the ideals of $\fg$ are in chain if and only if $\fg$ is a simple Lie algebra  or a direct sum of a nonzero nilpotent ideal $\fn$ and a simple algebra $\fs$ such that $\fn/\fn^2$ is a faithful $\fs$-module and $\fn^j/\fn^{j+1}$ are irreducible $\fs$-modules for $j\geq 1$. In that case, if $t$ is the nilindex of $\fn$, the ideals of $\fg$ are the $(t+1)$-element chain $0=\fn^t\subsetneq \fn^{t-1}\subsetneq \dots \subsetneq \fn^i\subsetneq \dots \subsetneq \fn\subsetneq \fg$.
\end{theorem}

In order to obtain a Lie algebra as described in previous theorem we need a triad $(\fs,\fn,\rho)$ where $\fs$ is simple, $\fn$ nilpotent and $\rho\colon \fs\to \der \fn$. So $\fn=m_1\oplus m_2\oplus \dots \oplus m_t$ is a direct sum of irreducible $\fs$-modules $m_i\cong \fn^i/\fn^{i+1}$. The nilpotency of $\fn$ makes the construction easier because of the terms in the lower central series are characteristic ideals (i.e., $\fn^i$ is $\der \fn$-invariant for all $i$) and $\fn$ is generated by any subspace $V$ such that $\fn=V\oplus \fn^2$. We also note that $\fg=\fs\oplus_\rho \fn$ is indecomposable, so $\rho$ is faithful. Even more:

\begin{theorem}[Snobl, 2010]\label{thm:basiconLevi}
Let $\fg$ be an indecomposable Lie algebra with product $[x,y]$, nilpotent radical $\fn$ of $(t+1)$-nilindex and nontrivial Levi decomposition $\fg=\fs\oplus \fn$ for some semisimple Lie algebra $\fs$. Then, there exists a decomposition of $\fn$ into a direct sum of $\fs$-modules.  
\begin{equation*}
	\fn=m_1\oplus m_2\oplus \dots \oplus m_t
\end{equation*}
where $\fn^j=m_j\oplus \fn^{j+1}$, $m_j\subseteq [m_1,m_{j-1}]$ such that $m_1$ is a faithful $\fs$-module and for $2\leq j\leq t$, $m_j$ decomposes into a sum of some subset of irreducible components of the tensor representation $m_1\otimes m_{j-1}$.
\end{theorem}

From Theorems \ref{thm:basiconPi} and \ref{thm:basiconLevi}, we get the following general construction of mixed algebras with chained lattices of ideals

\begin{theorem}\label{thm:chain}
	Let $\fs$ be a simple Lie algebra and $m_1$, $m_2$, $\dots$, $m_t$ irreducible $\fs$-modules with representations $\rho_i\colon \fs \to \mathfrak{gl}(m_i)$ for $i=1,\dots,t$ being $\rho_1$ faithful. Also, we have $\fs$-module homomorphisms
	\begin{equation*}
		p_{ijk}\colon m_i \otimes m_j \to m_k
	\end{equation*}
	where $1 \leq i \leq j \leq k \leq t$ and $i+j \leq k$ such that
	\begin{itemize}
		\item $p_{ijk}$ is skew-symmetric when $i=j$,
		\item $p_{ijk}$ is not null when $i=1$ and $k=1+j$
	\end{itemize}
	which also verify the identity
\begin{multline}\label{eq:thm_jac}
	\sum_{l=j+k}^{t-i} \sum_{r=i+l}^t p_{ilr}(u, p_{jkl}(v,w)) - 
	\sum_{l=i+k}^{t-j} \sum_{r=j+l}^t p_{jlr}(v, p_{ikl}(u,w)) \\+ 
	\sum_{l=i+j}^{t-k} \sum_{r=k+l}^t \hat{p}_{klr}(w, p_{ijl}(u,v)) = 0,
\end{multline}
where $u \in m_i$, $v \in m_j$ and $w \in m_k$ for $i+j+k \leq t$ and $i \leq j \leq k$. Here $\hat{p}_{klr} = p_{klr}$ if $k\leq l$, or $\hat{p}_{klr} = -p_{lkr}$ otherwise.

	The vector space $\fg = \fs \oplus m_1 \oplus m_2 \oplus \dots \oplus m_t$ with product
	\begin{align*}
		[s,s']_\fg&=[s,s']_{S},\\
		[s,u]_\fg &= \rho_i(s)(u),\\
		[u,v]_\fg &= \sum_{k=i+j}^t p_{ijk}(u,v)
	\end{align*}
	for $s, s' \in \fs$, $u \in m_i$, $v \in m_j$ and $i\leq j$ gives a Lie algebra with $(t+2)$-chain lattice of ideals
	\begin{equation*}
		0 < m_t < m_{t-1} \oplus m_t < \ldots < m_1 \oplus m_2 \oplus \dots \oplus m_t < \fg.
	\end{equation*}
	Moreover, every mixed Lie algebra with $(t+2)$-chain lattice of ideals has this form.
\end{theorem}
\begin{proof}
By definition $\fg$ is skew-symmetric and satisfies Jacobi identity, as it involves the usual product in $\fs$, some of its representations $\rho_i$, or it is imposed by condition in equation~\eqref{eq:thm_jac}, which is effectively Jacobi inside $\fn=m_1 \oplus \dots \oplus m_t$. Since $m_k$ is irreducible, any (nonzero) map $p_{1jk}$ is surjective when $k=j+1$. Then, from $[\fs,m_i]_\fg\subseteq m_i$ and $[m_i,m_j]_\fg=0$ for $i+j\geq t+1$, it is straightforward to check that $\fn$ is a nilpotent ideal with $k$-th lower central term 
\begin{equation*}
    \fn^k=\underset{s\geq k}{\oplus}m_s.
\end{equation*}
In fact $\fn$ is the only maximal ideal of $\fg$ because of $\rho_1$ is faithful and irreducible and $\fs$ is a simple Lie algebra. The ideals we obtain, and the reason why every non-solvable chain has this form is obtained from Theorem~\ref{thm:basiconPi}.
\end{proof}

\begin{remark}
    Lie algebras $\fg=\fs\oplus \fn$ described in Theorem~\ref{thm:chain} are perfect algebras ($\fg = \fg^2$) with nilpotent solvable radical, $\fn = m_1 \oplus m_2 \oplus \dots \oplus m_t$ with LCS terms $\fn^k = m_k \oplus \dots m_t$ and $[\fn^i, \fn^j]\subseteq \fn^{i+j}$. In addition, the summands $m_i$ are irreducible and either $\dim m_i=1$, so $[\fs, m_i] =0$, or $[\fs, m_i] =m_i$, and $m_{j+1} \subseteq [m_1, m_j]$. In particular, the module $m_1$ generates $\fn$ as a subalgebra and $m_2\subseteq \Lambda^2m_1$ by skew commutativity.
\end{remark}

Example~\ref{so-chain} and equation~\eqref{eq:n2t-4-5-cadena} in Section~\ref{s:generalities} follow the rules of the decompositions given in Theorem~\ref{thm:chain} by using the simple Lie algebras $\fsl(\fm)$ and $\mathfrak{so}(\fm)$ and taking $m_1=\fm$ the irreducible natural module and irreducible quotients of $\Lambda^2 \fm$ and $\fm\otimes\Lambda^2\fm$ (here $\Lambda^3\fm$ must be removed).
In each example, the homomorphisms $p_{ijk}$ are given by the projections inside tensor product modules. These examples are particular cases of a more general situation. According to \cite[Theorem 3.5]{benito2013levi} and Theorem~\ref{thm:basiconPi}, the mixed Lie algebras with nilradical of type $d$ and nilindex $t+1$ in which the lattice of ideals is a $n$-element chain are of the form
\begin{equation*}
    \mathcal{L}_t(\fs,\fm,\mathfrak{i})=\fs\oplus_{\id}\frac{\fn_{d,t}}{\mathfrak{i}},
\end{equation*}
where $\fs$ is a simple subalgebra of the Levi subalgebra of derivations of the free nilpotent $\fn_{d,t}$ (see \cite[Section 3]{benito2013levi} for a complete description of $\der \fn_{d,t}$) such that $\fn_{d,t}/\fn_{d,t}^2$ is a $\fs$-irreducible and faithful module, $\mathfrak{i}$ is an ideal, and also a $\fs$-submodule of $\fn_{d,t}^2$, that properly contains $\fn_{d,t}^t$. And each $\fs$-quotient module
\begin{equation*}
    \frac{\fn_{d,t}^k+\mathfrak{i}}{\fn_{d,t}^{k+1}+\mathfrak{i}},
\end{equation*}
for $2\leq k\leq t$, is $\fs$-irreducible. Explicit expressions for the irreducible blocks $m_i$ or the general product in Theorem~\ref{thm:chain} for the Lie algebra $\mathcal{L}_t(\fs,\fm,\mathfrak{i})$ are not easy to get, even in low nilindex. In the case of the 3-dimensional split simple Lie algebra, where the irreducible modules can be described in terms of differential operators and the maps $p_{ijk}$ are given by using partial differentiation of polynomials, computational algorithms with a detailed description (including bases and bracket products) of the algebras can be implemented.

\subsection[\texorpdfstring{$\mathfrak{sl}_2(\mathbb{K})$}{sl2}-modules and transvections]{$\bm{\mathfrak{sl}_2(\mathbb{K})}$-modules and transvections}\label{ss:sl-mod-trans}

As our final aim is constructing chains like the ones in Theorem~\ref{thm:chain} where $\fs = \mathfrak{sl}_2(\mathbb{K})$ we are going to see some arithmetic particularities of this algebra. They will play a significant role in obtaining these algebras automatically and theoretically. We follow ideas and tools from \cite{Dixmier_1984} and \cite{bremner2004invariant}.

Let $\mathbb{K}[x,y]$ be the ring of polynomials in the variables $x$ and $y$. For every $d$ greater or equal than 0, we denote as $V_d=\spa\langle x^d, x^{d-1}y,...,xy^{d-1}, y^d \rangle$ the set of homogeneous polynomials of degree $d$. Abusing notation, we will write $\deg V_d = d$.
Then, $V_d$ are vector spaces of dimension $d+1$, with $V_0= \mathbb{K}\cdot 1$. The set $V_d$ can also be viewed as a $\mathfrak{sl}_2(\mathbb{K})$-module in a natural way once $\mathfrak{sl}_2(\mathbb{K})$ is identified, into the Lie algebra $\mathfrak{gl}(\mathbb{K}[x,y])$, as the Lie subalgebra of partial derivations 
\begin{equation}\label{eq:partial-der-sl2}
	\spa\left\langle 
	e=x\,\frac{\partial}{\partial y},\enspace
	f=y\,\frac{\partial}{\partial x},\enspace
	h=x\,\frac{\partial}{\partial x}- y\,\frac{\partial}{\partial y}
	\right\rangle.
\end{equation}
	
This action turns $V_d$ into an irreducible module as seen in Figure~\ref{fig:arrow}. Even more, any finite-dimensional irreducible module of $\mathfrak{sl}_2(\mathbb{K})$ can be viewed in this way, being $V_0$ the trivial module.
\begin{figure}[!ht]
	\begin{tikzpicture}
	[	
		scale=0.95,
		arrowleft/.style ={<-, shorten <= 0.00cm, shorten >= 0.00cm, line width=0.25mm},
		arrowright/.style={->, shorten <= 0.00cm, shorten >= 0.00cm,
		line width=0.25mm}
	]
		
		\node (A0) at (0*1.5, 0) {$0$};
		\node (A1) at (1*1.5, -1.2) {$x^d$};
		\node (A2) at (2*1.5, -1.2) {$x^{d-1}y$};
		\node (A3) at (3*1.5, -1.2) {$x^{d-2}y^2$};
		\node (AM) at (4.25*1.5, 0) {$\ldots$};
		\node (A4) at (5.5*1.5, -1.2) {$x^2y^{d-2}$};
		\node (A5) at (6.5*1.5, -1.2) {$xy^{d-1}$};
		\node (A6) at (7.5*1.5, -1.2) {$y^d$};
		\node (A7) at (8.5*1.5, 0) {$0$};
		\coordinate (S) at (0, 0.2);
		
		\foreach \d in {1,2,3,5.5,6.5,7.5}
		   \filldraw [black] (1.5*\d, 0) circle (2pt);
		
		\foreach \d in {1,2,3,4,5.5,6.5,7.5}
		   \node at (1.5*\d-0.75, -0.65) {$e$};
		
		\foreach \d in {1,2,3,4.5,5.5,6.5,7.5}
		   \node at (1.5*\d+0.75, 0.7) {$f$};
		
		\foreach \d in {1,2,3,5.5,6.5,7.5}
		   \node at (1.5*\d, 1.65) {$h$};
		
		\foreach \d in {2,3,4,5.5,6.5,7.5,8.5}
			\draw[arrowleft] (1.5*\d-0.2,0.1) arc (30:150:1.7232/2-0.2);
			
		\foreach \d in {1,2,3,4,5.5,6.5,7.5}
			\draw[arrowright] (1.5*\d-0.2,-0.1) arc (-30:-150:1.7232/2-0.2);
		
		\foreach \d in {1,2,3,5.5,6.5,7.5}
			\draw[arrowleft, rotate=-90] (-0.4, 1.5*\d+0.15) arc (20:350:0.5);
	\end{tikzpicture}
	\caption{Diagram representing the $\mathfrak{sl}_2$-action over module $V_d$.}
	\label{fig:arrow}
\end{figure}
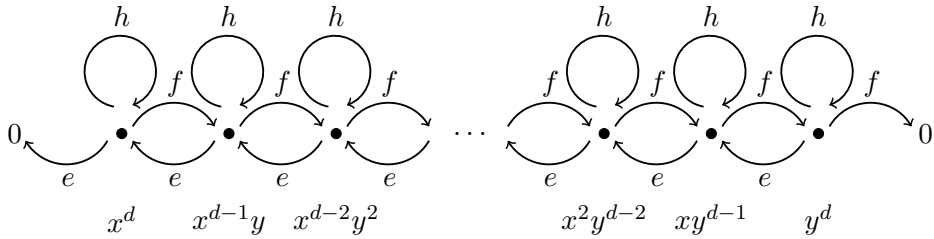

The \emph{Clebsch-Gordan's formula} gives the following decomposition of the tensor product of two $\mathfrak{sl}_2(\mathbb{K})$-irreducible modules. For $n \geq m$ we have
\begin{equation}\label{eq:clebsc-gordan}
V_n\otimes V_m\cong V_m\otimes V_n \cong \bigoplus_{k=0}^m V_{n+m-2k} \cong V_{n+m}\oplus V_{n+m-2}\oplus\dots\oplus V_{n-m}.
\end{equation}
While, when $n=m$ we can decompose
\begin{equation*}
\Lambda^2 V_n
\cong
\bigoplus_{k=0}^{\lfloor \frac{n-1}{2} \rfloor} V_{2n-4k-2}
\cong
V_{2n-2}\oplus V_{2n-6}\oplus V_{2n-10}\oplus\dots,
\end{equation*}
which is simply taking the odd $k$-summands in equation~\eqref{eq:clebsc-gordan}.

Now, for $0\leq k\leq \min(n,m)$, let us consider the bilinear \emph{transvection map} introduced in~\cite{Dixmier_1984} as $(\cdot,\cdot)_k\colon V_n\times V_m \to V_{n+m-2k}$ where
\begin{equation*}
(f,g)_k=
\frac{(m-k)!}{m!}\,
\frac{(n-k)!}{n!}\,
\sum_{i=0}^k\,(-1)^i
\binom{k}{i}\,
\frac{\partial^kf}{\partial x^{k-i}\partial y^i}\,
\frac{\partial^kg}{\partial x^{i}  \partial y^{k-i}}
\end{equation*}
We will use these transvections to define $\mathfrak{sl}_2(\mathbb{K})$-invariant products as it is explained in \cite{Dixmier_1984}. From Schur's Lemma and Clebs-Gordan's formula, it is easy to prove the following result: 
\begin{lemma}\label{lem:transvections}
Any bilinear $\mathfrak{sl}_2(\mathbb{K})$-invariant product $P_{n,m,p}\colon V_n\otimes V_m\to V_p$, satisfies:
\begin{itemize}
\item $P_{n,m,p}=\alpha\cdot (f,g)_k$ for some $\alpha \in \mathbb{K}$ when $p= n+m-2k$ and $0\leq k\leq \min\{n,m\}$. Here $P_{n,m,p}(b,a)=(-1)^kP_{n,m,p}(a,b)$.
\item $P_{n,m,p}=0$ otherwise.

\end{itemize}
So the product $P_{n,m,p}$ is either symmetric or skew-symmetric.
\end{lemma}

\begin{proof}
    Note that the set of $\mathfrak{sl}_2(\mathbb{K})$-invariant products $P_{m,n,p}$ are just the vector space of module homomorphisms $Hom_{\fsl_2(\mathbb{K})}(V_n\otimes V_m, V_p)$. The dimension of this set is equal to the number of copies of the irreducible $V_p$ inside $V_n\otimes V_m$. According to Clebs-Gordan's formula, the dimension is at most 1.
\end{proof}
This lemma plays an important role in the construction of Lie algebras in which their Levi factor is, up to isomorphism, $\mathfrak{sl}_2(\mathbb{K})$. 

\subsection{Algorithms}\label{sub:algorithm}

Now we have all the tools to start constructing Lie algebras whose ideals are in a chain. First, note we will go back to notation $p_{ijk}\colon m_i \otimes m_j \to m_k$, instead of the one Lemma~\ref{lem:transvections}, as it will be more convenient. Aside, as every $\mathfrak{sl}_2$-module $m_i$ can be identified by an integer, our algorithm will receive integers. But, instead of integers referring to the dimension or degree of each module we will set the integers in the following way: $n_1,n_2,\dots,n_t$ will define modules $m_1,m_2,\ldots,m_t$ where
\begin{equation}\label{eq:mi}
	m_i = V_{i\cdot n_1 - 2\sum_{j=2}^i n_j} = V_{i\cdot n_1 - 2n_2 - \ldots - 2n_i}
\end{equation}
So $m_1 = V_{n_1}$, $m_2 = V_{2n_1-2n_2}$, $m_3 = V_{3n_1-2n_2-2n_3}$ and so on. Here
\begin{multline*}
	\dim m_i = \dim m_1 + \dim m_{i-1} - 2n_i - 1 \\= i\cdot n_1 - 2\sum_{j=2}^i n_j + 1 = i\cdot n_1 - 2n_2 - \ldots - 2n_i + 1.
\end{multline*}

As this type of algebras would appear constantly till the end of the article, we would introduce the following definition:
\begin{definition}
	We say a Lie algebra $\fg$ is a $\mathfrak{sl}_2$-chained Lie algebra of length $t+2$ when $\fg$ has Levi decomposition $\fs \oplus \fn$ with $\fs \cong \mathfrak{sl}_2$ and it is form as in Theorem~\ref{thm:chain}.
 \end{definition}
 We will denote these $\mathfrak{sl}_2$-chained Lie algebras as
	\begin{equation*}
	\mathfrak{C}( \{m_i\}_{i=1,\dots,t },\{\alpha_{ijk}\}_{i=1,\dots,\lfloor\frac{t}{2}\rfloor;\, j=i,\dots,t-1;\, k=i+j,\dots,t}),
	\end{equation*}
	where $m_i$ will be the $\mathfrak{sl}_2$-modules of the form $V_{k_i}$ for some $k_i$ and
	\begin{equation*}
		p_{ijk} = \alpha_{ijk}(\cdot, \cdot)_{c_{ijk}}
	\end{equation*}
	where
	\begin{align*}
		c_{ijk} &= \frac{\dim m_i + \dim m_j - \dim m_k - 1}{2}\\
		  &= \frac{\deg m_i + \deg m_j - \deg m_k}{2} \\
		  &= \left(\frac{i+j-k}{2}\right) n_1 - \sum_{l=2}^i n_l + \sum_{l=j+1}^k n_l\\
		  &= \left(\frac{i+j-k}{2}\right) n_1 - n_2 - \dots -n_i + n_{j+1} + \dots + n_k.
	\end{align*}
	Therefore, our chains will be of the form
	\begin{equation*}
		\fg=\mathfrak{sl}_2 \oplus m_1 \oplus m_2 \oplus \dots \oplus m_t,
	\end{equation*}
	with $t+2$ ideals in a chain: $0$, $\bigoplus_{i=k}^t m_i$ for $k=1,\dots,t$; and $\fg$. They will also have the following Lie bracket definition ($s,s'\in \mathfrak{sl}_2$, $u\in m_i$):
\begin{itemize}
	\item Inside $\mathfrak{sl}_2$, where the elements are viewed as partial differentiation maps according to equation~\eqref{eq:partial-der-sl2}, the definition is given by the using usual special linear Lie bracket $[s, s']=ss'-s's$.
	\item The product $[\mathfrak{sl}_2, m_i]$ is defined using the representation $\rho_i$,
	\begin{equation*}
		[s,u] = \rho_i(s)(u)
	\end{equation*}
	where for the standard basis in $\mathfrak{sl}_2$ it is defined as
	\begin{align*}
		\rho(e) &= x\,\frac{\partial}{\partial y},\\
		\rho(f) &= y\,\frac{\partial}{\partial x},\\
		\rho(h) &= x\,\frac{\partial}{\partial x}- y\,\frac{\partial}{\partial y}.
	\end{align*}
	\item The product between the modules satisfies a couple of conditions:
	\begin{align*}
		[m_1, m_i] &\supseteq m_{i+1}, \\
		[m_i, m_j] &\subseteq \sum_{k=i+j}^t m_k = m_{i+j} + m_{i+j+1} + \dots + m_t
	\end{align*}
	Given $u \in m_i$ and $v \in m_j$ where $i\leq j$ the Lie product is
	\begin{equation*}
		[u,v] = \sum_{k=i+j}^t p_{ijk}(u,v)
	\end{equation*}
	where $p_{ijk}\colon m_i \otimes m_j \to m_k$ such that
	\begin{equation*}
		p_{ijk}(u,v) = \alpha_{ijk} \cdot (u,v)_{c_{ijk}}
	\end{equation*}
	If $c_{ijk}\notin \mathbb{Z}_{\geq 0}$ then, as stated in Lemma~\ref{lem:transvections}, $\alpha_{ijk} = 0$. Moreover
	\begin{itemize}
		\item $\alpha_{ijk} \neq 0$ for $i=1$ and $k=j+1$,
		\item As $p_{iik}$ must be skew-symmetric then $c_{iik}$ must be odd or $\alpha_{iik} = 0$.
	\end{itemize}
\end{itemize}

So, when $t=3$ we will have chains $\mathfrak{C}(\{m_1,m_2,m_3\}, \{\alpha_{112}, \alpha_{113},\alpha_{123}\})$, and for $t=4$ chains will be
\begin{equation*}
	\mathfrak{C}(\{m_1,m_2,m_3,m_4\}, \{\alpha_{112}, \alpha_{113},\alpha_{114},\alpha_{123},\alpha_{124},\alpha_{134},\alpha_{224}\}).
\end{equation*}

Now, all these tools and notation, in combination with Theorem~\ref{thm:chain}, open up the idea to develop Algorithm~\ref{alg:main} to find all these algebras given the already mentioned list of integers referring to the irreducible modules.
\begin{algorithm}\label{alg:main}
	Checks if a succession of $t$ integers referring to $t$ modules is compatible with being a Lie algebra whose ideals are in chain of $t+2$ ideals.

\noindent\textbf{Input}: The algorithm receives $t$ integers $n_1$, $n_2$, $\dots$, $n_t$ referring to $\mathfrak{sl}_2$-modules as defined by equation~\eqref{eq:mi}.

\noindent\textbf{Output}: A boolean value (true or false) indicating if there is a  Lie algebra $\mathfrak{C}(\{m_1,m_2,\dots, m_t\},\{\alpha_{ijk}\}_{i=1,\dots, \lfloor\frac{t}{2}\rfloor;\, j=i,\dots,t-1;\, k=i+j,\dots,t})$ for some $\alpha_{ijk}$. And the list those valid $\alpha_{ijk}$ such the chained algebra exists.

In case the validity of the algebra is subject to only some values of the parameters $\alpha_{ijk}$ the algorithm also gives them.

\noindent\textbf{Steps:} The algorithm is divided in two main steps:
\begin{enumerate}
	\item Check integers input: this is equivalent to checking if 
	\begin{itemize}
		\item $m_2 \subseteq \Lambda^2 m_1$
		\item $m_{i+1} \subseteq m_1 \otimes m_i$ for every $i = 2, \dots, t$
	\end{itemize}
	In terms of integers, this translates into $n_2$ being an odd number, and for $i=1,\dots,t-1$,
	\begin{equation*}
		0 \leq n_{i+1} \leq \min(n_1, \dim m_{i}-1) = \min\left(n_1, i\cdot n_1 - 2\sum_{j=2}^{i} n_j\right)
	\end{equation*}
	\item Check Jacobi identities inside $m_1 \oplus m_2 \oplus \dots \oplus m_t$: We need to study $J(u,v,w)=0$ for $u \in m_i$, $v \in m_j$ and $w \in m_k$ such $i+j+k \leq t$ and $i \leq j \leq k$. As seen in Theorem~\ref{thm:chain}, here
\begin{multline}\label{eq:jacobiComplete}
	J(u,v,w) = [u, [v,w]] + [v, [w,u]] + [w, [u,v]] \\ = 
	\sum_{l=j+k}^{t-i} \sum_{r=i+l}^t p_{ilr}(u, p_{jkl}(v,w)) - 
	\sum_{l=i+k}^{t-j} \sum_{r=j+l}^t p_{jlr}(v, p_{ikl}(u,w)) \\  + 
	\sum_{l=i+j}^{t-k} \sum_{r=k+l}^t \hat{p}_{klr}(w, p_{ijl}(u,v)) = 0,
\end{multline}
where $\hat{p}_{klr} = p_{klr}$ if $k\leq l$, or $\hat{p}_{klr} = -p_{lkr}$ otherwise.
\end{enumerate}
\end{algorithm}

In case we want to find every tuple $(n_1, n_2, \dots, n_t)$ that gives a chain we should call Algorithm~\ref{alg:main} using at least all integers that satisfy step 1 in Algorithm~\ref{alg:main}.

Note, Algorithm~\ref{alg:main} only works for every chain of more than 3 ideals (so $t> 2$). But we do not need more in case we want to study smaller cases. A 1-chain is simply the zero-dimensional Lie algebra, the only 2-chain of this form is the simple Lie algebra $\mathfrak{sl}_2$ with no modules. And, even for cases as 3-chains of 4-chains ($t=1,2$) the algorithm is unnecessary as the algebras are $\mathfrak{sl}_2 \oplus V_n$ for every $n \geq 0$, and $\mathfrak{sl}_2 \oplus V_n \oplus V_{2n-2k}$ for $1 \leq k \leq n$ and $k$ being an odd number. In all those algebras, skew-symmetry is guaranteed by transvection properties, while the Jacobi identity is trivially null in the modules. It is only in larger cases when the use of the algorithm becomes relevant for finding valid algebras. That is why, for cases $t=3$ and $t=4$ we can find a detailed implementation in Algorithm~\ref{alg:5chain} and Algorithm~\ref{alg:6chain} respectively.

\begin{algorithm}\label{alg:5chain}
	Detailed implementation of Algorithm~\ref{alg:main} in case $t=3$.

\noindent\textbf{Input}: Three integers $n_1$, $n_2$ and $n_3$ referring to modules
\begin{itemize}
	\item $m_1 = V_{n_1}$,
	\item $m_2 = V_{2n_1-2n_2}$,
	\item $m_3 = V_{3n_1-2n_2-2n_3}$.
\end{itemize}

\noindent\textbf{Output}: A boolean value (true or false) indicating if there is a  Lie algebra $\mathfrak{C}(\{m_1,m_2,m_3\},\{\alpha_{112}, \alpha_{113}, \alpha_{123}\})$ for any $\alpha_{ijk}$ such that $\alpha_{112} \cdot\alpha_{123} \neq 0$.

\noindent\textbf{Steps}: The algorithm is divided in two main steps:
\begin{enumerate}
	\item Check integers input:
	\begin{itemize}
		\item $1 \leq n_2 \leq n_1$ and $n_2$ is odd,
		\item $0 \leq n_3 \leq \min(n_1, \dim m_2-1) = \min(n_1, 2n_1 - 2n_2)$.
	\end{itemize}
	\item Check Jacobi identity inside the nilradical $N = m_1 \oplus m_2 \oplus m_3$. As we have $t=3$, equation~\eqref{eq:jacobiComplete} can be simplified as we have only one possibility: $i = j= k = 1$. Therefore, the only identity to check is 
\begin{equation*}
	p_{123}(u, p_{112}(v, w)) +
	p_{123}(v, p_{112}(w, u)) +
	p_{123}(w, p_{112}(u, v)) = 0.
\end{equation*}
for every $u,v,w \in m_1$. Every term of this equality has both coefficients $\alpha_{123}$ and $\alpha_{112}$, and, as they are not zero, we can simplify obtaining
\begin{equation}\label{eq:jac3}
	(u, (v, w)_{n_2})_{n_3} +
	(v, (w, u)_{n_2})_{n_3} +
	(w, (u, v)_{n_2})_{n_3} = 0.
\end{equation}
If this is true $\mathfrak{C}(\{m_1,m_2,m_3\},\{\alpha_{112}, \alpha_{113}, \alpha_{123}\})$ is a Lie algebra for any $\alpha_{ijk}$ such that $\alpha_{112} \cdot\alpha_{123} \neq 0$.
\end{enumerate}
\end{algorithm}

Applying Algorithm~\ref{alg:5chain}, we can obtain every possible 5-chain for $n_1 \leq 32$. This way we observe two families: the first family are 4 chains that exist for every $n_1$ and are described in Table~\ref{tab:5chainsG}. The second family are algebras which repeat every four $n_1$ values. This last family appears in Table~\ref{tab:5chainsM}. And, up to\footnote{The $n_1=32$ cap is set arbitrarily to limit the computational cost, and it is not based on some property that does not work for greater values.} $n_1=32$ these are the only chains of 5 ideals. Therefore, it is quite probable that this extends for every $n_1$.

\begin{table}[h!]
\centering
\renewcommand\arraystretch{1.2}
\begin{tabular}{|ccc|c|c|}\hline
$n_1$&$n_2$&$n_3$& $\mathfrak{sl}_2$-modules    & Condition \\\hline
 $n$ & 1 & 0 & $V_{n} \oplus V_{2n-2} \oplus V_{3n-2}$ & $n\geq 1$\\\hline
 $n$ & 1 & 1 & $V_{n} \oplus V_{2n-2} \oplus V_{3n-4}$ & $n\geq 2$\\\hline
 $n$ & 1 & 3 & $V_{n} \oplus V_{2n-2} \oplus V_{3n-8}$ & $n\geq 3$\\\hline
 $n$ & 3 & 1 & $V_{n} \oplus V_{2n-6} \oplus V_{3n-8}$ & $n\geq 4$\\\hline
\end{tabular}
\caption{Chains of 5 ideals for $n_1 \leq 32$ which repeat for every $n_1$.}
\label{tab:5chainsG}
\end{table}

\begin{table}[h!]
\centering
\renewcommand\arraystretch{1.2}
\begin{tabular}{|P{1.5cm}P{1.75cm}P{1.5cm}|c|c|}\hline
$n_1$&$n_2$&$n_3$& $\mathfrak{sl}_2$-modules    & Condition \\\hline
$4n$   & $2n+1$ & $4n-3$ & $V_{4n}   \oplus V_{4n-2} \oplus V_{4}$ & $n\geq 2$\\\hline
$4n+1$ & $2n+1$ & $4n$   & $V_{4n+1} \oplus V_{4n}   \oplus V_{1}$ & $n\geq 0$\\\hline
$4n+2$ & $2n+1$ & $4n+1$ & $V_{4n+2} \oplus V_{4n+2} \oplus V_{2}$ & $n\geq 0$\\\hline
$4n+3$ & $2n+1$ & $4n+3$ & $V_{4n+3} \oplus V_{4n+4} \oplus V_{1}$ & $n\geq 0$\\\hline
$4n+4$ & $2n+1$ & $4n+3$ & $V_{4n+4} \oplus V_{4n+6} \oplus V_{4}$ & $n\geq 0$\\\hline
\end{tabular}
\caption{Chains of 5 ideals for $n_1 \leq 32$ which repeat every four $n_1$.}
\label{tab:5chainsM}
\end{table}

\begin{algorithm}\label{alg:6chain}
	Detailed implementation of Algorithm~\ref{alg:main} in case $t=4$.

\noindent\textbf{Input}: Four integers $n_1$, $n_2$, $n_3$ and $n_4$ referring to modules
\begin{itemize}
	\item $m_1 = V_{n_1}$,
	\item $m_2 = V_{2n_1-2n_2}$,
	\item $m_3 = V_{3n_1-2n_2-2n_3}$,
	\item $m_4 = V_{4n_1-2n_2-2n_3-2n_4}$.
\end{itemize}

\noindent\textbf{Output}: A boolean value (true or false) indicating if there is a  Lie algebra $\mathfrak{C}(\{m_1,m_2,m_3, m_4\},\{\alpha_{ijk}\}_{i=1,2;\, j=i,\dots,3;\, k=i+j,\dots,4})$ for some $\alpha_{ijk}$. It also returns a value $\alpha$ which gives the following scalar restriction over $\alpha_{224}$:
\begin{equation}\label{eq:alpha}
	\alpha = \frac{\alpha_{224}\cdot \alpha_{112}}{\alpha_{123}\cdot \alpha_{134}}.
\end{equation}
Any values for $\alpha_{112}, \alpha_{123}, \alpha_{134}$ different from zero would give an algebra, at least when considering $\alpha_{124} = \alpha_{113} = 0$.

\noindent\textbf{Steps}: The algorithm is divided in two main steps:
\begin{enumerate}
	\item Check integers input:
	\begin{itemize}
		\item $1 \leq n_2 \leq n_1$ and $n_2$ is odd,
		\item $0 \leq n_3 \leq \min(n_1, \dim m_2-1) = \min(n_1, 2n_1 - 2n_2)$,
		\item $0 \leq n_4 \leq \min(n_1, \dim m_3-1) = \min(n_1, 3n_1 - 2n_2 - 2n_3)$.
	\end{itemize}
	\item Check Jacobi identity inside the nilradical $N = m_1 \oplus m_2 \oplus m_3 \oplus m_4$. Here $t=4$, so equation~\eqref{eq:jacobiComplete}, taking into account $p_{114}$ and $p_{124}$ go to the centre, appears in two scenarios:
	\begin{enumerate}
		\item Three elements $u,v,w \in m_1$.
		\begin{multline*}
			p_{123}(u, p_{112}(v, w)) +
			p_{123}(v, p_{112}(w, u)) +
			p_{123}(w, p_{112}(u, v)) \\+
			p_{124}(u, p_{112}(v, w)) + 
 			p_{124}(v, p_{112}(w, u)) +
 			p_{124}(w, p_{112}(u, v)) \\+ 
 			p_{134}(u, p_{113}(v, w)) +
 			p_{134}(v, p_{113}(w, u)) + 
 			p_{134}(w, p_{113}(u, v)) = 0.
		\end{multline*}
		Taking projections, we can separate the first three addends from the rest into two null equations. The first one, following the same procedure as in Algorithm~\ref{alg:5chain} turns again into equation~\eqref{eq:jac3}. While the second part could be omitted taking $\alpha_{124} = 0$ and $\alpha_{113} = 0$.
		\item Two elements $u,v \in m_1$ and another $w \in m_2$. 
		\begin{equation}\label{eq:cond42}
			p_{134}(u, p_{123}(v, w)) -
			p_{134}(v, p_{123}(u, w)) +
  			p_{224}(w, p_{112}(u, v)) = 0.
		\end{equation}
		For this equation we have two options:
		\begin{itemize}
			\item If $n_3 + n_4 - n_2$ is even or negative, or $n_2 + n_3 + n_4 > 2n_1$ then $\alpha_{224} = 0$, and equation~\eqref{eq:cond42} turns into
			\begin{equation*}
				(u,(v,w)_{n_3})_{n_4} -
				(v,(u,w)_{n_3})_{n_4} = 0
			\end{equation*}
			\item If not, equation~\eqref{eq:cond42} turns into
			\begin{equation*}
				(u,(v,w)_{n_3})_{n_4} -
				(v,(u,w)_{n_3})_{n_4} + \alpha\cdot
				(w,(u,v)_{n_2})_{n_3+n_4-n_2} = 0
			\end{equation*}
			for $\alpha = \frac{\alpha_{224}\alpha_{112}}{\alpha_{123}\alpha_{134}}$. Note this $\alpha$ is unique as in this case we have $(w, (u,v)_{n_2})_{n_3+n_4-n_2} \neq 0$.
		\end{itemize}
		
	\end{enumerate}	
\end{enumerate}
\end{algorithm}

On the same way, applying Algorithm~\ref{alg:6chain}, we can try to obtain every possible module combination of 6-chains for $n_1 \leq 32$. Again, we can distinguish two groups. The general one that repeats for every $n_1$ which appears in Table~\ref{tab:chains6G}. But, in contrast to what happens with five ideals, in this case there are some chains that only work for some $n_1$ values. These particular cases are listed in Table~\ref{tab:chains6E}.
\begin{table}[h!]
\centering
\renewcommand\arraystretch{1.5}
\begin{tabular}{|cccc|c|c|c|}\hline
$n_1$&$n_2$&$n_3$&$n_4$& $\mathfrak{sl}_2$-modules                           &$\alpha$&Condition\\[0.1cm]\hline
$n$& 1 & 0 & 0 & $V_{n} \oplus V_{2n-2} \oplus V_{3n-2} \oplus V_{4n-2}$ & 0      & $n \geq 1$\\[0.1cm]\hline
$n$& 1 & 0 & 2 & $V_{n} \oplus V_{2n-2} \oplus V_{3n-2} \oplus V_{4n-6}$ & $\frac{4(4n-3)}{3(3n-2)}$& $n \geq 2$\\[0.1cm]\hline
$n$& 1 & 1 & 1 & $V_{n} \oplus V_{2n-2} \oplus V_{3n-4} \oplus V_{4n-6}$ & $\frac{2n-2}{3n-4}$     & $n \geq 2$\\[0.1cm]\hline
\end{tabular}
\caption{Chains of 6 ideals for $n_1 \leq 32$ which repeat for every $n_1$.}
\label{tab:chains6G}
\end{table}

\begin{table}[h!]
\centering
\renewcommand\arraystretch{1.2}
\begin{tabular}{cccccc}
$\bm{n_1}$&$\bm{n_2}$&$\bm{n_3}$&$\bm{n_4}$& $\bm{\mathfrak{sl}_2}$\textbf{-modules}                     &$\bm{\alpha}$\\\hline
 3 & 1 & 1 & 3 & $V_{3} \oplus V_{4}  \oplus V_{5}  \oplus V_{2} $ & 7/5    \\
 3 & 1 & 3 & 1 & $V_{3} \oplus V_{4}  \oplus V_{1}  \oplus V_{2} $ & -2     \\\hline
 4 & 1 & 3 & 3 & $V_{4} \oplus V_{6}  \oplus V_{4}  \oplus V_{2} $ & 3/2    \\
 4 & 3 & 1 & 3 & $V_{4} \oplus V_{2}  \oplus V_{4}  \oplus V_{2} $ & 1/2    \\\hline
 5 & 1 & 3 & 5 & $V_{5} \oplus V_{8}  \oplus V_{7}  \oplus V_{2} $ & 22/21  \\
 5 & 3 & 1 & 5 & $V_{5} \oplus V_{4}  \oplus V_{7}  \oplus V_{2} $ & 12/7   \\\hline
 6 & 1 & 3 & 5 & $V_{6} \oplus V_{10} \oplus V_{10} \oplus V_{6} $ & 1      \\
\end{tabular}
\caption{Chains of 6 ideals for $n_1 \leq 32$ which do not repeat for $n_1$.}
\label{tab:chains6E}
\end{table}

\begin{remark}
In case we want to study what happens in $(t+2)$-chains for $t\geq 5$ we can use the general Algorithm~\ref{alg:main}. It is important to note, that in these cases the complexity increases rapidly. For instance, when $t=5$, which is the simplest case, checking Jacobi identity following equation~\eqref{eq:jacobiComplete} produces up to 4 cases to study: 
\begin{enumerate}
	\item Three elements in $m_1$
	\item Two elements in $m_1$ and other in $m_2$
	\item Two elements in $m_1$ and other in $m_3$
	\item Two elements in $m_2$ and other in $m_1$
\end{enumerate}
As seen in Algorithm~\ref{alg:5chain} and Algorithm~\ref{alg:6chain}, many $\alpha_{ijk}$ could be considered null and we could simplify them. But, as $\alpha_{224}$ could be different from zero many more subcases appear making it much harder to solve.
\end{remark}

\section{Lie algebra structure existence}
\label{s:lemmas}

Although our algorithms are useful for finding chains, they do not let us generalize and prove the existence of 5 or 6-chains whose dimension is as big as we want. But, at least, seeing their results we know approximately where we should look.

Before proving some results of existence, we need to introduce Gordan identities, which appear in~\cite{Dixmier_1984} (see also \cite{bremner2004invariant} for further information). These are some relationships that transvections fulfil which will be helpful during proofs.
	
	\begin{definition}[Gordan's Identity]\label{def:gordan}
		Let $f\in V_n$, $g\in V_m$ and $h\in V_p$, and let $\alpha_1$, $\alpha_2$ and $\alpha_3$ be non-negative integers such that $\alpha_1 + \alpha_2 \leq p$, $\alpha_2 + \alpha_3 \leq m$, $\alpha_3 + \alpha_1 \leq n$, with $\alpha_1 = 0$ o $\alpha_2+\alpha_3 = m$. Then
		\begin{equation*}
			\begin{bmatrix}
				f & g & h\\
				m & n & p\\
				\alpha_1 & \alpha_2 & \alpha_3
			\end{bmatrix}
			= 0,
		\end{equation*}
		where
		\begin{align*}
		\begin{split}
			\begin{bmatrix}
				f & g & h\\
				m & n & p\\
				\alpha_1 & \alpha_2 & \alpha_3
			\end{bmatrix}
			&=
			\sum_{i\geq 0} \frac{\binom{n-\alpha_1-\alpha_3}{i}\binom{\alpha_2}{i}}{\binom{m+n-2\alpha_3-i+1}{i}} ((f,g)_{\alpha_3+i},h)_{\alpha_1+\alpha_2-i} \\
			&+ (-1)^{\alpha_1+1}
			\sum_{i\geq 0} \frac{\binom{p-\alpha_1-\alpha_2}{i}\binom{\alpha_3}{i}}{\binom{m+p-2\alpha_2-i+1}{i}} ((f,h)_{\alpha_2+i},g)_{\alpha_1+\alpha_3-i}.
		\end{split}
		\end{align*}
	\end{definition}

\subsection{Chains with five ideals}

Now, we have all the necessary tools to prove all chains in Tables~\ref{tab:5chainsG} and \ref{tab:5chainsM} can be extended for any $n$ and not only up to 32. But before proving the results we introduce the following simplified notation for Gordan's Identity. When writing $[f,g,h,n,\alpha_1, \alpha_2, \alpha_3]$ we refer to (for $f,g,h \in V_n$)
	 	\begin{equation*}
			\begin{bmatrix}
				f & g & h\\
				n & n & n\\
				\alpha_1 & \alpha_2 & \alpha_3
			\end{bmatrix}\ \text{and}
		\end{equation*}
	    \begin{equation*}
	        [f,g,h,n,\alpha_1,\alpha_2,\alpha_3]^* =\sum_{\overset{\circlearrowright}{f,g,h}}[f,g,h,n,\alpha_1,\alpha_2,\alpha_3]-\sum_{\overset{\circlearrowright}{g,f,h}}[g,f,h,n,\alpha_1,\alpha_2,\alpha_3].
	    \end{equation*}

		From now on $\rho_{(n_1,\dots, n_t)}=\rho_{n_1}\oplus \dots \oplus \rho_{n_t}$ will denote the direct sum representation of $\fsl_2(\ku)$ on the vector space module $W(n_1,\dots, n_t)=V_{n_1}\oplus V_{2n_1-2n_2}\oplus\dots\oplus V_{tn_1-2\sum_{i=2}^tn_i}$. According to Section \ref{ss:sl-mod-trans}, the action $\rho_{n_j}$ on the module $V_{jn_1-2\sum_{k=2}^jn_k}$ of homogeneous polynomials of degree $jn_1-2\sum_{k=2}^jn_k$, is given in terms of differential operators. The Lie bracket that makes $W(n_1,\dots, n_t)$ into a nilpotent Lie algebra is defined rescaling by $\alpha_{ijk}\in \ku$ the $c_{ijk}$-transvection $(f,g)_{c_{ijk}}\in V_{kn_1-2\sum_{s=2}^kn_s}$ of $f\in V_{in_1-2\sum_{r=2}^in_r}$ and $g\in V_{jn_1-2\sum_{q=2}^jn_q}$. Here $i=1,\dots, \lfloor\frac{t}{2}\rfloor$, $j=i,\dots, t-1$ and $k=i+j, \dots, t$. In the particular case where $i=1$, $j=p$ and $k=p+1$
	\begin{equation*}
	    [f,g]_W= \alpha_{ijk}(f,g)_{n_{k}}\  \text{and}\  \alpha_{ijk}\neq 0.
	\end{equation*}
	We will also denote the tuple $\lambda_{(n_1,\dots, n_t)}=(\alpha_{ijk})_{ijk}$. This encodes the structure constants of the Lie algebra $W(n_1,\dots,n_t)$. Here, $\alpha_{ijk}$ is nonzero for all $i=1, j=p=1,\dots t-1$ and $k=p+1$. We shall refer to a fold $(\alpha_{ijk})_{ijk}$ structure constants fold of the Lie algebra
    \begin{equation*}
	    \fg_{(n_1,\dots, n_t)}^{\rho,\lambda}=\fsl_2(\ku)\oplus_\rho W(n_1,\dots, n_t)_\lambda\quad \text{and} \quad \begin{cases}
     \rho=\rho_{(n_1,\dots, n_t)},\\ \lambda=\lambda_{(n_1,\dots, n_t)}.\end{cases}
	\end{equation*}
\begin{proposition}\label{prop:lie-table1}
   The $3$-tuples $(n,1,0)$, $(n,1,1)$, $(n,1,3)$ and $(n,3,1)$ generate the following $\lambda$-parametric families of $\fsl_2$-chained Lie algebras:
   
    \begin{enumerate}[\quad a)]
        \item $\fg^{\rho,\lambda}_{(n,1,0)}=\fsl_2(\ku)\oplus_\rho (V_n\oplus V_{2n-2}\oplus V_{3n-2})_\lambda$, for $n\geq 1$,
        
        \item $\fg^{\rho,\lambda}_{(n,1,1)}=\fsl_2(\ku)\oplus_\rho (V_n\oplus V_{2n-2}\oplus V_{3n-4})_\lambda$, for $n\geq 2$,
        
        \item $\fg^{\rho,\lambda}_{(n,1,3)}=\fsl_2(\ku)\oplus_\rho (V_n\oplus V_{2n-2}\oplus V_{3n-8})_\lambda$, for $n\geq 3$,
        
        \item $\fg^{\rho,\lambda}_{(n,3,1)}=\fsl_2(\ku)\oplus_\rho (V_n\oplus V_{2n-6}\oplus V_{3n-8})_\lambda$, for $n\geq 4$,
    \end{enumerate}
    where scalar threefolds $\lambda_{(n,k,j)} = (\alpha_{112}, \alpha_{113},\alpha_{123})$ determine the product in nilradical by $\alpha_{112}(\cdot,\cdot)_k$, $\alpha_{113}(\cdot,\cdot)_{\frac{2k+2j-n}{2}}$ and $\alpha_{123}(\cdot,\cdot)_j$. And they take values:
    \begin{alignat*}{2}
    \lambda_{(n,1,0)}&=(\alpha_{112}, 0,\alpha_{123})&&\\
    \lambda_{(n,1,1)}&=(\alpha_{112}, 0,\alpha_{123}) \text{, for } n\neq 2,\qquad & 
    \lambda_{(2,1,1)}&=(\alpha_{112}, \alpha_{113},\alpha_{123})\\
    \lambda_{(n,1,3)}&=(\alpha_{112}, 0,\alpha_{123}) \text{, for } n\neq 6,\qquad & 
    \lambda_{(6,1,3)}&=(\alpha_{112}, \alpha_{113},\alpha_{123})\\
    \lambda_{(n,3,1)}&=(\alpha_{112}, 0,\alpha_{123}) \text{, for } n\neq 6,\qquad & 
    \lambda_{(6,3,1)}&=(\alpha_{112}, \alpha_{113},\alpha_{123})\\
    \end{alignat*}
\end{proposition}
\begin{proof}
    We denote the different tuples in the general form $(n_1,n_2,n_3)$, so $m_1=V_{n_1}$, $m_2=V_{2n_1-2n_2}$ and $m_3=V_{3n_1-2n_2-2n_3}$. It is a straightforward computation that the module summand $W(n_1,n_2,n_3)=m_1\oplus m_2\oplus m_3$ is as described in all four items. Using Clebsch-Gordan’s formula from equation~\eqref{eq:clebsc-gordan} we can observe $m_2$ appears in $\Lambda^2 m_1$ and $m_3$ appears in $m_1 \otimes m_2$. Therefore, by construction, we only need to check Jacobi identity from equation~\eqref{eq:jac3}, for every $f,g,h \in m_1$. And this equality can be proved using Gordan's Identity from Definition~\ref{def:gordan} on expressions:
    \begin{align*}
    [f,g,h,n,0,0,1]&=((f,g)_1,h)_0-((f,h)_0,g)_1-\frac{1}{2} ((f,h)_1,g)_0,\\
    [f,g,h,n,0,1,1]&=((f,g)_1,h)_1+\frac{1}{2}((f,g)_2,h)_0 - ((f,h)_1,g)_1 - \frac{1}{2} ((f,h)_2,g)_0,\\
    [f,g,h,n,0,2,2]&=((f,g)_2,h)_2 + ((f,g)_3,h)_1 + \frac{(n-2)(n-3)}{(2n-5)(2n-6)}((f,g)_4,h)_0\\ & \quad -((f,h)_1,g)_1 -((f,h)_3,g)_1 - \frac{(n-2)(n-3)}{(2n-5)(2n-6)}((f,h)_4,g)_0,\\
    [f,g,h,n,0,1,3]&=((f,g)_3,h)_1 + \frac{1}{2}((f,g)_4,h)_0 - ((f,g)_1,h)_3 - \frac{3}{2}((f,h)_2,g)_2\\ &\quad  -\frac{3(n-1)}{2(2n-3)}((f,h)_3,g)_1 - \frac{(n-1)}{4(2n-5)}((f,h)_4,g)_0.
    \end{align*}
    Depending on the different tuples, equation~\eqref{eq:jac3} is equivalent to the following identities (note that $(a,b)_k=(-1)^k(b,a)_k$ according to Lemma \ref{lem:transvections}):
    \begin{itemize}
        \item $(n_1,n_2,n_3)=(n,1,0)$ for $n \geq 1$: equation~\eqref{eq:jac3} follows from,
            \begin{equation*}
		    [f,g,h,n,0,0,1] - [h,g,f,n,0,0,1] = 0.
	        \end{equation*}
	   We note that $\alpha_{113}=0$ because $m_3=V_{3n-2}$ is not contained in $\Lambda^2m_1$.
	    \item $(n_1,n_2,n_3)=(n,1,1)$ for $n \geq 2$: equation~\eqref{eq:jac3} follows from,
            \begin{equation*}
		    [f,g,h,n,0,1,1] + [g,f,h,n,0,1,1] + [h,g,f,n,0,1,1] = 0.
		    \end{equation*}
		    In this case, $m_3=V_{3n-4} \subset \Lambda^2m_1$ implies $3n-4=2n-2k$ with $k$ odd. Then $2k=4-n\geq 0$ and therefore $\alpha_{113}=0$ if $n\geq 3$ and $n=2$ implies $k=1$ and any $\alpha_{113}$ is valid. 
		\item $(n_1,n_2,n_3)=(n,1,3)$ for $n \geq 4$: equation~\eqref{eq:jac3} follows from,
		\begin{multline*}
			[f,g,h,n,0,1,3]^* - \frac{7n-9}{4n-6}
		   \Big([f,g,h,n,0,2,2]\\-
			[g,f,h,n,0,2,2]-
			[h,g,f,n,0,2,2]\Big)= 0.
		\end{multline*}
		And it is equivalent to $[f,g,h,3,1,1,2]^* = 0$ if $n=3$. Here, $\alpha_{113}=0$ if $n\neq 6$ and $n=6$ implies $k=1$ and any $\alpha_{113}$ is valid.
	   \item $(n_1,n_2,n_3)=(n,3,1)$ for $n \geq 4$: here equation~\eqref{eq:jac3} is equivalent to $	[f,g,h,n,0,2,2] + [g,f,h,n,0,2,2] + [h,g,f,n,0,2,2] = 0$. \qedhere
    \end{itemize}
\end{proof}

Most of previous information on the Proposition \ref{prop:lie-table1} can be originally found distributed in several lemmas in~\cite[Secci\'on 2.3.1]{ivan}, work that has been revisited, sorted and extended to produce the mentioned proposition. From the algorithms, we also reach the following series of Lie algebras.

\begin{proposition}\label{prop:lie-table2}
   The $3$-tuples $(4n,2n+1,4n-3)$, $(4n+1,2n+1,4n)$, $(4n+2,2n+1,4n+1)$, $(4n+3,2n+1,4n+3)$ and $(4n+4,2n+1,4n+3)$ generate only algebras with $\mathbb{N}^+$-graded nilradical. So, the valid scalar threefolds are $\lambda_{(n_1,n_2, n_3)}=(\alpha_{112},0,\alpha_{123})$. The resulting $\lambda$-parametric families of $\fsl_2$-chained Lie algebras are:
    \begin{enumerate}[\quad a)]
        \item $\fg^{\rho,\lambda}_{(4n,2n+1,4n-3)}=\fsl_2(\ku)\oplus_\rho (V_{4n}  \oplus V_{4n-2}\oplus V_4)_\lambda$ for $n\geq 2$.
        \item $\fg^{\rho,\lambda}_{(4n,2n+1,4n)}  =\fsl_2(\ku)\oplus_\rho (V_{4n+1}\oplus V_{4n}  \oplus V_1)_\lambda$ for $n\geq 0$.
        \item $\fg^{\rho,\lambda}_{(4n,2n+1,4n+1)}=\fsl_2(\ku)\oplus_\rho (V_{4n+2}\oplus V_{4n+2}\oplus V_2)_\lambda$ for $n\geq 0$.
        \item $\fg^{\rho,\lambda}_{(4n,2n+1,4n+3)}=\fsl_2(\ku)\oplus_\rho (V_{4n+3}\oplus V_{4n+4}\oplus V_1)_\lambda$ for $n\geq 0$.
        \item $\fg^{\rho,\lambda}_{(4n,2n+1,4n+3)}=\fsl_2(\ku)\oplus_\rho (V_{4n+4}\oplus V_{4n+6}\oplus V_4)_\lambda$ for $n\geq 0$.
    \end{enumerate}
    And products in the nilradicals are given by $\alpha_{112}(\cdot,\cdot)_{n_2}$ and $\alpha_{123}(\cdot,\cdot)_{n_3}$.
\end{proposition}

\begin{proof}
    We follow the notation introduced in the proof of  Proposition \ref{prop:lie-table1}. It is a straightforward computation that the module summand $W(n_1,n_2,n_3)=m_1\oplus m_2\oplus m_3$ is as described in all items in the list. Using Clebsch-Gordan’s formula from equation~\eqref{eq:clebsc-gordan} we check that $\alpha_{113}=0$ in all the cases, $m_2$ appears in $\Lambda^2 m_1$ and $m_3$ appears in $m_1 \otimes m_2$. So, to establish the result, it only remains to prove Jacobi identity from equation~\eqref{eq:jac3} for every $f,g,h \in m_1$.
    As in Proposition \ref{prop:lie-table1}, we proceed to check using Gordan identities:
    \begin{itemize}
        \item $(n_1,n_2,n_3)=(4n,2n+1,4n-3)$, for $n\geq 2$: equation~\eqref{eq:jac3} is just
	\begin{equation*}
		[f,g,h,4n,2n-2,2n+2,2n-2]^* -
	   2[f,g,h,4n,2n-2,2n+1,2n-1]^* = 0
	\end{equation*}
    \end{itemize}
        For the other cases $n\geq 0$ is fixed and, 
    \begin{itemize}
        \item $(n_1,n_2,n_3)=(4n+1,2n+1,4n)$: equation~\eqref{eq:jac3} follows from the identity $[f,g,h,4n+1,2n,2n+1,2n]^*=0$.
        
        \item $(n_1,n_2,n_3)=(4n+2,2n+1,4n+1)$: equation~\eqref{eq:jac3} is equivalent to $[f,g,h,4n+2,2n,2n+1,2n+1]^*=0$.
        
        \item $(n_1,n_2,n_3)=(4n+3,2n+1,4n+3)$: equation~\eqref{eq:jac3} is just
	\begin{multline*}
		[f,g,h,4n+3,2n-1,2n,2n+3]^*\\+\frac{4n+3}{2n}[f,g,h,4n+3,2n-1,2n+3,2n]^*=0
	\end{multline*}
	for $n \geq 1$. While case $n=0$ is just item c) for $n=3$ in Proposition~\ref{prop:lie-table1}.
        \item $(n_1,n_2,n_3)=(4n+4,2n+1,4n+3)$: equation~\eqref{eq:jac3} is obtained from \begin{multline*}
		[f,g,h,4n+4,2n-2,2n+4,2n]^* -
		[f,g,h,4n+4,2n-2,2n,2n+4]^* \\-
		\frac{3(n+2)(2n+3)(7n+10)}{4n(4n+1)(6 n+7)}
		\bigg(
		[f,g,h,4n+4,2n-2,2n+4,2n]^* \\-
		\frac{5(n+2)}{7n+10}
		[f,g,h,4n+4,2n-2,2n+3,2n+1]^* \bigg) = 0
	\end{multline*}
	for $n \geq 1$. While case $n=0$ is item c) for $n=4$ in Proposition~\ref{prop:lie-table1}.\qedhere
    \end{itemize}
\end{proof}

\begin{remark}
As seen in some proofs, some algebras are repeated in Propositions \ref{prop:lie-table1} and \ref{prop:lie-table2}. The cases $n=0$ from Proposition \ref{prop:lie-table2} in items b), c), d) and e) coincide with the cases from Proposition \ref{prop:lie-table1} in item a) for $n=1$, item b) for $n=2$, and item c) for $n=3,4$.
\end{remark}

\begin{remark}\label{rmk:no-cadena}
Other $3$-tuples that do not produce chain ideal Lie algebras are:
\begin{alignat*}{5}
&(n,1,2)_{n\geq 2},\quad
&(n,1,4)_{n\geq 4},\qquad
&(n,1,5)_{n\geq 5},\qquad
&(n,1,6)_{n\geq 6},\\
&(n,3,0)_{n\geq 3},\quad
&(2n,2n-1,1)_{n\geq 3},\qquad
&(2n,2n-1,0)_{n\geq 2},\qquad
&(n,3,2)_{n\geq 3},\\
&&(2n+1,2n+1,0)_{n\geq 1},\qquad
&(2n,2n-1,2)_{n\geq 3}.
\end{alignat*}
We can check why and where they do not work as chains in~\cite[Secci\'on 2.3.1]{ivan}.
\end{remark}

\begin{example}
The $3$-tuples $(n,1,2)_{n\geq 2}$ or $(2n+1,2n+1,0)_{n\geq 1}$ do not produce Lie algebra structures with chain ideal lattice. In the first case, the Lie product must be induced on $V_n\oplus V_{2n-2}\oplus V_{3n-6}$ by using $\lambda=(\alpha_{112},0,\alpha_{123})$. But Jacobi identity fails (unless $\alpha_{112}\alpha_{123}=0$):
    \begin{multline*}
    J(x^n,yx^{n-1},y^2x^{n-2})=\sum_{cyclic}((x^n,yx^{n-1})_1,y^2x^{n-2})_2=\\\frac{9n-12}{n^2(2n-3)(n-1)}x^{3n-6}
    \end{multline*}
For the second tuple, the vector space is $V_{2n+1}\oplus V_0\oplus V_{2n+1}$ and Jacobi identity also fails:
\begin{multline*}
J(yx^{2n},y^2x^{2n-1})=\sum_{cyclic}((yx^{2n},y^2x^{2n-1})_{2n+1},y^{2n}x)_0=\frac{1}{2n+1}y^2x^{2n-1}.
\end{multline*}
\end{example}

\subsection{Chains with 6 ideals}

Now, we have to prove the existence results inspired by Table~\ref{tab:chains6G}. Before proving them, we are going to introduce simplified notations as in all following results $m_1 = V_n$ and $m_2 = V_{2n-2}$. So, we will write $[h,f,g,n,\alpha_1, \alpha_2, \alpha_3]_1$, $[f,h,g,n,\alpha_1, \alpha_2, \alpha_3]_2$ and $[f,g,h,n,\alpha_1, \alpha_2, \alpha_3]_3$ instead of
	 	\begin{equation*}
			\begin{bmatrix}
				h & f & g\\
				2n-2 & n & n\\
				\alpha_1 & \alpha_2 & \alpha_3
			\end{bmatrix},
			\qquad
			\begin{bmatrix}
				f & h &g\\
				n & 2n-2 & n\\
				\alpha_1 & \alpha_2 & \alpha_3
			\end{bmatrix},
			\qquad
			\begin{bmatrix}
				f & g & h\\
				n & n & 2n-2 \\
				\alpha_1 & \alpha_2 & \alpha_3
			\end{bmatrix},
		\end{equation*}
	respectively, and
		\begin{align*}
		\begin{split}
			[f,g,h,n,\alpha_1,\alpha_2,\alpha_3]_\star &=
			[f,g,h,n,\alpha_1,\alpha_2,\alpha_3]_3
			-[g,f,h,n,\alpha_1,\alpha_2,\alpha_3]_3
			\\&
			+[g,h,f,n,\alpha_1,\alpha_2,\alpha_3]_2
			-[h,g,f,n,\alpha_1,\alpha_2,\alpha_3]_1
			\\&
			+[h,f,g,n,\alpha_1,\alpha_2,\alpha_3]_1
			-[f,h,g,n,\alpha_1,\alpha_2,\alpha_3]_2,
		\end{split}
		\end{align*}
		for $f,g \in V_n$ and $h\in V_{2n-2}$.
		
	Note, as first seen in Algorithm~\ref{alg:6chain}, every chained Lie algebra of length 6 depends on an $\alpha$ parameter which imposes restrictions over $\alpha_{224}$  for every not null $\alpha_{112}$, $\alpha_{123}$, $\alpha_{134}$ as seen in equation~\eqref{eq:alpha}.

\begin{proposition}\label{prop:lie-table3}
The $4$-tuples $(n,1,0,0), (n,1,0,2)$ and $(n,1,1,1)$ generate the following parametric families of $\fsl_2$-chained Lie algebras: 
\begin{enumerate}[\quad a)]
    \item $\fg^{\rho,\lambda}_{(n,1,0,0)}=\fsl_2(\ku)\oplus_\rho(V_{n}\oplus V_{2n-2}\oplus V_{3n-2}\oplus V_{4n-2})_\lambda$ for $n\geq 1$,
    \item $\fg^{\rho,\lambda}_{(n,1,0,2)}=\fsl_2(\ku)\oplus_\rho(V_{n}\oplus V_{2n-2}\oplus V_{3n-2}\oplus V_{4n-6})_\lambda$ for $n\geq 2$,
    \item $\fg^{\rho,\lambda}_{(n,1,1,1)}=\fsl_2(\ku)\oplus_\rho(V_{n}\oplus V_{2n-2}\oplus V_{3n-4}\oplus V_{4n-6})_\lambda$ for $n\geq 2$.
\end{enumerate}
Along here, $\lambda_{(n_1,n_2,n_3,n_4)}=(\alpha_{112}, 0, 0, \alpha_{123}, 0, \alpha_{134},\alpha_{224})$, where
\begin{alignat*}{2}
    \alpha_{224} &= 0 &\text{ when } \lambda &= \lambda_{(n,1,0,0)},\\
    \alpha_{224} &= \frac{4(4n-3)\,\alpha_{123}\,\alpha_{134}}{3(3n-2)\,\alpha_{112}}\quad &\text{ when } \lambda &= \lambda_{(n,1,0,2)},\\
    \alpha_{224} &= \frac{(2n-2)\,\alpha_{123}\,\alpha_{134}}{(3n-4)\,\alpha_{112}} &\text{ when } \lambda &= \lambda_{(n,1,1,1)},
\end{alignat*}
defines the product in the $\mathbb{N}$-graded nilradical given by $\alpha_{112}(\blank,\blank)_{n_2}$,  $\alpha_{123}(\blank,\blank)_{n_3}$,  $\alpha_{134}(\blank,\blank)_{n_4}$ and $\alpha_{224}(\blank,\blank)_{n_3+n_4-n_2}$. We also have the not necessarily graded particular cases, where $\alpha_{113}$, $\alpha_{114}$, $\alpha_{124}$ are not necessarily zero, given by the sevenfolds $\lambda_{(n,1,k,j)}$,
\begin{align*}
    \lambda_{(2,1,0,2)}&=(\alpha_{112}, 0, \alpha_{114}, \alpha_{123}, \alpha_{124}, \alpha_{134},\alpha_{224}),\\
    \lambda_{(4,1,0,2)}&=(\alpha_{112}, 0, 0,  \alpha_{123}, \alpha_{124},  \alpha_{134}, \alpha_{224}),\\
    \lambda_{(2,1,1,1)}&=(\alpha_{112}, \alpha_{113}, \alpha_{114}, \alpha_{123}, \alpha_{124}, \alpha_{134},\\
    \lambda_{(4,1,1,1)}&=(\alpha_{112}, 0, 0, \alpha_{123}, \alpha_{124}, \alpha_{134}, \alpha_{224}).
\end{align*}
For these four last parametric families, $\fg^{\rho,\lambda}_{(2,1,k,j)}$ and $\fg^{\rho,\lambda}_{(4,1,k,j)}$, the entry $\alpha_{224}$ and the irreducible decomposition are described as in the graded case, and the product in the nilradical is given by $\alpha_{112}(\cdot,\cdot)_1$, $\alpha_{11p}(\cdot,\cdot)_1$ for\footnote{In general, we have $\alpha_{113}(\cdot,\cdot)_{\frac{2k-n+2}{2}}$ and $\alpha_{114}(\cdot,\cdot)_{j+k-n+1}$. But, as we only have non null $\alpha_{113}$ and $\alpha_{114}$ for some $\lambda$ we can simplify to those which have them.} $p=3,4$, $\alpha_{123}(\cdot,\cdot)_k$, $\alpha_{124}(\cdot,\cdot)_{\frac{4-n}{2}}$, $\alpha_{134}(\cdot,\cdot)_j$, $\alpha_{224}(\cdot,\cdot)_1$.
\end{proposition}

\begin{proof}
According to Proposition \ref{prop:lie-table1} and the notation in its proof, it is easily check that the module summand $W(n_1,n_2,n_3,n_4)=m_1\oplus m_2\oplus m_3\oplus m_4$ is as described in the three items.

Assume first $(n_1,n_2,n_3,n_4)=(n,1,0,0)$ and look at scalar entries $\alpha_{113}$, $\alpha_{114}$, $\alpha_{124}$ and $\alpha_{224}$ of $\lambda_{(n,1,0,0)}$. Using Clebsch-Gordan’s formula from equation~\eqref{eq:clebsc-gordan} we check that $\alpha_{113}=\alpha_{114}=\alpha_{124} =\alpha_{224}=0$ because modules $m_3=V_{3n-2}$ and $m_4=V_{4n-2}$ are not contained in $\Lambda^2m_1$ and $m_4$ is not contained in either $m_1\otimes m_2$ or $m_2\otimes m_2$. Proposition~\ref{prop:lie-table1} says tuple $(n,1,0)$ generates a chain. So our case is reduced to checking if $m_4 = V_{4n-2}$ appears in $m_1 \otimes m_3$ decomposition, which it is true; and studying  Jacobi identity for two elements in $m_1$ and another in $m_2$. But this last condition, seen in equation~\eqref{eq:cond42}, is equivalent to $[f, g, h, n, 0, 0, 0]_3=0$ for $n\geq 1$, $f,g \in V_n$ and $h \in V_{2n-2}$.

Suppose now $(n_1,n_2,n_3, n_4)=(n,1,0,2)$, and note that $\alpha_{113}=0$ and $\alpha_{114}\neq 0$ (respectively $\alpha_{124}\neq 0$) only if $n=2$ (respectively $n=2,4$). For the $\mathbb{N}$-graded condition $\alpha_{113}=\alpha_{114}=\alpha_{124}=0$, Proposition~\ref{prop:lie-table1} says that  tuple $(n,1,0)$ generates a chain. So this case is reduced to checking if $m_4 = V_{4n-6}$ appears in $m_1 \otimes m_3$ decomposition, which it is true; and studying Jacobi identity for two elements in $m_1$ and another in $m_2$. But this last condition, seen in equation~\eqref{eq:cond42}, is equivalent to
		\begin{multline*}
			[h, f, g, n, 0, 2, 0]_1 +
			[f, h, g, n, 0, 2, 0]_2 - 
			[h, g, f, n, 0, 2, 0]_1 \\- 
			[g, h, f, n, 0, 2, 0]_2 +
			\frac{14n-18}{9n-12}\big(
			[f, g, h, n, 0, 2, 0]_3 \\-
			[g, f, h, n, 0, 2, 0]_3\big) +
			\frac{(n-1)(2n-4)}{(3n-4)(3n-2)} G = 0
		\end{multline*}
		for $n\geq 2$, $f,g \in V_n$ and $h \in V_{2n-2}$, where
		\begin{equation*}
			G=\frac{2n-3}{n-1}[h,f,g,n,0,1,1]_1+[g,h,f,n,0,1,1]_2+[f,g,h,n,0,1,1]_3,
		\end{equation*}
		which will appear again in our final case
		$(n_1,n_2,n_3, n_4)=(n,1,1,1,1)$. By reapplying  Proposition~\ref{prop:lie-table1}, tuple $(n,1,1)$ generates a chain. So, assuming $\alpha_{113}=\alpha_{114}=\alpha_{124}=0$, our case is reduced to checking if $m_4 = V_{4n-6}$ appears in $m_1 \otimes m_3$ decomposition, which it is true; and studying Jacobi identity for two elements in $m_1$ and another in $m_2$. But this last condition, seen in equation~\eqref{eq:cond42}, is equivalent to 
		\begin{equation*}
			\frac{2n-3}{n-1}[h,f,g,n,0,1,1]_1+[g,h,f,n,0,1,1]_2+[f,g,h,n,0,1,1]_3 = 0,
		\end{equation*}
		for $n\geq 2$, $f,g \in V_n$ and $h\in V_{2n-2}$. Here also appear the particular tuples $(2,1,1,1)$ and $(4,1,1,1)$ for which $(\alpha_{113},\alpha_{114}, \alpha_{12,4})\neq (0,0,0)$ and $\alpha_{113}=\alpha_{114}=0$ but $(\alpha_{113},\alpha_{114}, \alpha_{124})\neq 0$. Exceptions $(2,1,0,1)$, $(4,1,1,1)$, $(2,1,1,1)$ and $(4,1,1,1)$ are covered by the particular sevenfold $\lambda$ at the end of the proposition.
\end{proof}

\subsection{Overview}

Finally, to sum up, we will see the relation among all the previously described chains. As on any $t$-chain we can take the quotient by their last $\mathfrak{sl}_2$-modules to obtain smaller chains, there is a strong relation between chains. This is idea is expressed in the following lemma.
\begin{lemma}
The tuple $(n_1, n_2, \dots, n_t)$ could form a $\mathfrak{sl}_2$-chained Lie algebra if $(n_1, n_2, \dots, n_k)$ forms a valid $\mathfrak{sl}_2$-chained Lie algebra for every $k \leq t$.
\end{lemma}
\begin{proof}
	If $(n_1, n_2, \dots, n_t)$ forms a valid $\mathfrak{sl}_2$-chained Lie algebra $L = \mathfrak{sl}_2 \oplus N$ for $N = m_1 \oplus m_2 \oplus \dots \oplus m_t$, then $L/N^{k+1} = \mathfrak{sl}_2 \oplus m_1 \oplus m_2 \oplus \dots \oplus m_k$ would be a Lie algebra for every $k$.
\end{proof}
This is the same as saying that, if $(n_1, n_2, \dots, n_k)$ does not produce any valid chain, then $(n_1, n_2, \dots, n_k, \dots, n_t)$ would never produce a valid chain.
    \begin{example}
    From Remark \ref{rmk:no-cadena}, tuples $(n,1,2,*)_{n\geq 2}$ or $(2n+1,2n+1,0,*)_{n\geq 1}$ do not produce Lie algebras with $t$-chain ideal lattice for $t\geq 4$.
    \end{example}
This result is interesting for creating a tree-dependency between these chains, which can be seen in Figure~\ref{fig:tree}.

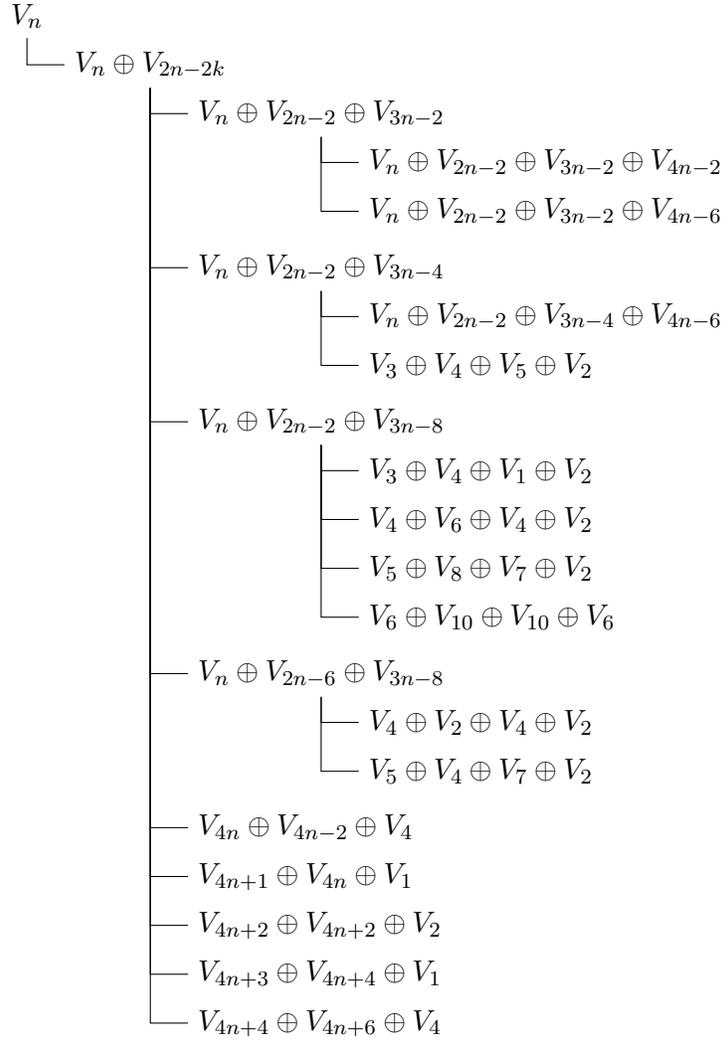
\begin{figure}[!ht]
\centering
\begin{tikzpicture}
[   
    level 1/.style = {},
    level 2/.style = {},
    level 3/.style = {},
    level 4/.style = {},
    every node/.append style = {anchor = west},
    grow via three points={one child at (0.5,-0.65) and two children at (0.5,-0.65) and (0.5,-1.3)},
    edge from parent path={(\tikzparentnode\tikzparentanchor) |- (\tikzchildnode\tikzchildanchor)},
]
\node{$V_{n}$}
	child{node{$V_{n} \oplus V_{2n-2k}$}
		child{node{$V_{n} \oplus V_{2n-2} \oplus V_{3n-2}$}
			child{node{$V_{n} \oplus V_{2n-2} \oplus V_{3n-2} \oplus V_{4n-2}$}}
			child{node{$V_{n} \oplus V_{2n-2} \oplus V_{3n-2} \oplus V_{4n-6}$}}
		}
		child{node[yshift = -0.65*2cm-0.1cm]{$V_{n} \oplus V_{2n-2} \oplus V_{3n-4}$}
			child{node{$V_{n} \oplus V_{2n-2} \oplus V_{3n-4} \oplus V_{4n-6}$}}
			child{node{$V_{3} \oplus V_{4} \oplus V_{5} \oplus V_{2}$}}
		}
		child{node[yshift = -0.65*4cm-0.2cm]{$V_{n} \oplus V_{2n-2} \oplus V_{3n-8}$}
			child{node{$V_{3} \oplus V_{4} \oplus V_{1} \oplus V_{2}$}}
			child{node{$V_{4} \oplus V_{6} \oplus V_{4} \oplus V_{2}$}}
			child{node{$V_{5} \oplus V_{8} \oplus V_{7} \oplus V_{2}$}}
			child{node{$V_{6} \oplus V_{10} \oplus V_{10} \oplus V_{6}$}}
		}
		child{node[yshift = -0.65*8cm-0.3cm]{$V_{n} \oplus V_{2n-6} \oplus V_{3n-8}$}
			child{node{$V_{4} \oplus V_{2} \oplus V_{4} \oplus V_{2}$}}
			child{node{$V_{5} \oplus V_{4} \oplus V_{7} \oplus V_{2}$}}
		}
		child{node[yshift = -0.65*10cm-0.4cm]{$V_{4n  } \oplus V_{4n-2} \oplus V_{4}$}}
		child{node[yshift = -0.65*10cm-0.4cm]{$V_{4n+1} \oplus V_{4n  } \oplus V_{1}$}}
		child{node[yshift = -0.65*10cm-0.4cm]{$V_{4n+2} \oplus V_{4n+2} \oplus V_{2}$}}
		child{node[yshift = -0.65*10cm-0.4cm]{$V_{4n+3} \oplus V_{4n+4} \oplus V_{1}$}}
		child{node[yshift = -0.65*10cm-0.4cm]{$V_{4n+4} \oplus V_{4n+6} \oplus V_{4}$}}
	};
\end{tikzpicture}
\caption{Relationships of $\mathfrak{sl}_2$-chained Lie algebras up to length 6.}
\label{fig:tree}
\end{figure}

\section{Final comments}\label{FC}
The notion of \emph{quasi-cyclic or homogeneous or Carnot} Lie algebra is introduced in 1963 by G. Leger \cite{leger1963derivations}. Up to isomorphisms, any  quasi-cyclic Lie algebra is a quotient of a free nilpotent Lie algebra by some  homogeneous ideal. The variety of quasi-cyclic Lie algebras includes free nilpotent, generalised Heisenberg and filiforms among others Lie algebras. Even more, quasi-cyclic Lie algebras of type $d$ are the class of nilpotent Lie algebras that contain a minimal generator set $\{e_1,\dots, e_d\}$ such that the correspondence $e_i\mapsto e_i$ extends to a derivation of $\fn$ (see~\cite[Corollary 1]{johnson1975homogeneous}). Note that such a derivation is invertible.

An automorphism of a real Lie algebra is called \emph{expanding authomorphism} if it is semisimple  with eigenvalues greater than $1$ in absolute value. According to \cite{dyer1970nilpotent}, quasi-cyclic Lie algebras admits expanding automorphisms, but the converse is false. In fact, real quasi-cyclic Lie algebras are those Lie algebras that admit \emph{grading automorphisms} \cite{johnson1975homogeneous}. And following \cite[Theorem 3.1 and Theorem 3.3]{dere2017gradings} (see also \cite{cornulier2016gradings}), the class of real Lie algebras admitting expanding automorphisms is just the class of positive graded Lie algebras (positive naturally graded along this paper) which is bigger than the quasi-cyclic class. In the realm of nilpotent Lie groups, the existence of an expanding map (respectively a non-trivial self-cover) in an infra-nilmanifold modeled on a Lie group $G$ is equivalent to the fact that the real algebra $\text{Lie}(G)$ admits a positive grading (respectively a naturally and non-trivial grading). Expanding automorphisms of real Lie algebras are hyperbolic (maps without eigenvalues $\pm 1$), and Lie algebras admitting \emph{hyperbolic automorphisms} are nilpotent (see~\cite[Proposition 3.6]{smale1967differentiable}).

 Any $3$-step nilpotent Lie algebra can be endowed with a \emph{complete affine structure} (check~\cite{scheuneman1974affine}). In the context of Lie algebras this concept is equivalent to that of \emph{left-symmetric structure} (see~\cite{elduque1994transitive} and references therein). Any naturally graded real Lie algebra admits a left symmetric structure according to \cite[Theorem 3.1]{dekimpe2003expanding}.
 
 Along the paper we have built series of naturally graded Lie algebras of nilindex $4$ and $5$ whose derivation algebra contains a subalgebra isomorphic to $\fsl_2(\ku)$. The algebras are given by $t$-tuples $(n_1,\dots, n_t)$ for $t=3,4$ that encode their structure in an easy way: $$(n_1+1,2n_1-2n_2+1, \dots, 2n_1-2(n_2+\dots +n_t)+1)$$ is their general type and the transvections $\alpha_{1ij}(\cdot,\cdot)_{n_{i+1}}$ where $j=i+1$ let us describe their Lie bracket product. Among these algebras we point out $V_1\oplus V_0\oplus V_1$, general type $(2,1,2)$, that corresponds to $\fn_{2,3}$ and $V_1\oplus V_0\oplus V_1\oplus V_2$, general type $(2,1,2,3)$, and its mixed extension $\fsl_2(\ku)\oplus V_1\oplus V_0\oplus V_1\oplus V_2$. The first and the third algebras admit a non-degenerate invariant symmetric bilinear form (\emph{metric Lie structure}). We also note that their decompositions into irreducible modules follow a symmetric pattern which is caused by the existence of an invariant form. The second one is not a metric Lie algebra. The series $V_2\oplus \dots \oplus V_2$ ($n$-summands) of general type given by the $n$-tuple $(3,3,\dots,3)_{n\geq 1}$ and their mixed extensions $\fsl_2(\ku)\oplus V_2\oplus \dots \oplus V_2$ are metric Lie algebras (rescaling of structure constants may be required). These algebras are a particular case of the series of current Lie algebras $\fg_n(S)$, for $S$ simple, which are naturally graded and metrizable. The general type of $\fn(\fg_n(S))$ is $(m,\dots, m)_{n\geq 1}$, where $m$ is the dimension of the simple Lie algebra $S$.

\section*{Funding}
The authors have been supported by research grant MTM2017-83506-C2-1-P of `Ministerio de Econom\'ia, Industria y Competitividad, Gobierno de Espa\~na' (Spain) until 2022 and by grant PID2021-123461NB-C21, funded by MCIN/AEI/10.13039/501100011033 and by “ERDF A way of making Europe” since then. J. Rold\'an-L\'opez was also supported by a predoctoral research grant FPI-2018 of `Universidad de La Rioja'.

\bibliographystyle{apalike}
\bibliography{bibliography}

\begin{thebibliography}{}

\bibitem[Alvarez and Cartes, 2019]{alvarez2019cohomology}
Alvarez, M.~A. and Cartes, F. (2019).
\newblock Cohomology and deformations for the {H}eisenberg {H}om-{L}ie
  algebras.
\newblock {\em Linear and Multilinear Algebra}, 67(11):2209--2229.

\bibitem[Benito, 1992]{pilar1992lie}
Benito, P. (1992).
\newblock Lie algebras in which the lattice formed by the ideals is a chain.
\newblock {\em Communications in Algebra}, 20(1):93--108.

\bibitem[Benito and {de-la-Concepci{\'o}n}, 2013]{benito2013levi}
Benito, P. and {de-la-Concepci{\'o}n}, D. (2013).
\newblock On {L}evi extensions of nilpotent {L}ie algebras.
\newblock {\em Linear Algebra and its Applications}, 439(5):1441--1457.

\bibitem[Benito and Rold\'{a}n-L\'{o}pez, 2020]{Benito_2020}
Benito, P. and Rold\'{a}n-L\'{o}pez, J. (2020).
\newblock Lie algebras with a finite number of ideals.
\newblock {\em Linear and Multilinear Algebra}.

\bibitem[Bremner and Hentzel, 2004]{bremner2004invariant}
Bremner, M. and Hentzel, I. (2004).
\newblock Invariant nonassociative algebra structures on irreducible
  representations of simple {L}ie algebras.
\newblock {\em Experimental Mathematics}, 13(2):231--256.

\bibitem[Cornulier, 2016]{cornulier2016gradings}
Cornulier, Y. (2016).
\newblock Gradings on {L}ie algebras, systolic growth, and cohopfian properties
  of nilpotent groups.
\newblock {\em Bulletin de la Soci\'{e}t\'{e} Math\'{e}matique de France},
  144(4):693–744.

\bibitem[Dekimpe and Lee, 2003]{dekimpe2003expanding}
Dekimpe, K. and Lee, K.~B. (2003).
\newblock Expanding maps, anosov diffeomorphisms and affine structures on
  infra-nilmanifolds.
\newblock {\em Topology and its Applications}, 130(3):259--269.

\bibitem[Der{\'e}, 2017]{dere2017gradings}
Der{\'e}, J. (2017).
\newblock Gradings on lie algebras with applications to infra-nilmanifolds.
\newblock {\em Groups, Geometry, and Dynamics}, 11(1):105--120.

\bibitem[Dilworth, 2009]{dilworth2009decomposition}
Dilworth, R.~P. (2009).
\newblock {\em A decomposition theorem for partially ordered sets}, pages
  139--144.
\newblock Springer.

\bibitem[Dixmier, 1984]{Dixmier_1984}
Dixmier, J. (1984).
\newblock Certaines alg\`{e}bres non associatives simples d\'{e}finies par la
  transvection des formes binaires.
\newblock {\em Journal f\"{u}r die reine und angewandte Mathematik},
  346:110–128.

\bibitem[Dyer, 1970]{dyer1970nilpotent}
Dyer, J.~L. (1970).
\newblock A nilpotent {L}ie algebra with nilpotent automorphism group.
\newblock {\em Bulletin of the American Mathematical Society}, 76(1):52--56.

\bibitem[Elduque and Myung, 1994]{elduque1994transitive}
Elduque, A. and Myung, H.~C. (1994).
\newblock On transitive left-symmetric algebras.
\newblock In {\em Non-associative algebra and its applications}, pages
  114--121. Springer.

\bibitem[Gauger, 1973]{gauger_1973}
Gauger, M.~A. (1973).
\newblock On the classification of metabelian {L}ie algebras.
\newblock {\em Transactions of the American Mathematical Society},
  179:293–293.

\bibitem[Humphreys, 1997]{Humphreys_1997}
Humphreys, J.~E. (1997).
\newblock {\em Introduction to Lie algebras and representation theory}.
\newblock Graduate texts in mathematics. Springer, New York, 7th corr. print
  edition.

\bibitem[Jacobson, 1979]{Jacobson_1979}
Jacobson, N. (1979).
\newblock {\em Lie Algebras}.
\newblock Dover Books on Advanced Mathematics. Dover Publications, 1st edition
  by this publisher, corrected printing edition.

\bibitem[Johnson, 1975]{johnson1975homogeneous}
Johnson, R.~W. (1975).
\newblock Homogeneous {L}ie algebras and expanding automorphisms.
\newblock {\em Proceedings of the American Mathematical Society},
  48(2):292--296.

\bibitem[Kharraf, 2021]{kharraf2021classification}
Kharraf, Y.~E. (2021).
\newblock Classification problem of simple {H}om-{L}ie algebras.
\newblock {\em arXiv preprint arXiv:2101.11518}.

\bibitem[Leger, 1963]{leger1963derivations}
Leger, G. (1963).
\newblock Derivations of {L}ie algebras {III}.
\newblock {\em Duke Mathematical Journal}, 30(4):637--645.

\bibitem[Mattarei, 2022]{mattarei2022constituents}
Mattarei, S. (2022).
\newblock Constituents of graded lie algebras of maximal class and chains of
  thin {L}ie algebras.
\newblock {\em Communications in Algebra}, 50(2):726--739.

\bibitem[P\'erez-Aradros, 2016]{ivan}
P\'erez-Aradros, I. (2016).
\newblock {\'A}lgebras de {L}ie con ret\'iculo de ideales en cadena.
\newblock Bachelor's thesis at `{U}niversidad de {L}a {R}ioja'.
\newblock
  \url{https://investigacion.unirioja.es/documentos/5e4a8591299952031e843bb1}.

\bibitem[Scheuneman, 1974]{scheuneman1974affine}
Scheuneman, J. (1974).
\newblock Affine structures on three-step nilpotent lie algebras.
\newblock {\em Proceedings of the American Mathematical Society},
  46(3):451--454.

\bibitem[Smale, 1967]{smale1967differentiable}
Smale, S. (1967).
\newblock Differentiable dynamical systems.
\newblock {\em Bulletin of the American mathematical Society}, 73(6):747--817.

\bibitem[\v{S}nobl, 2010]{Snobl_2010}
\v{S}nobl, L. (2010).
\newblock On the structure of maximal solvable extensions and of {L}evi
  extensions of nilpotent {L}ie algebras.
\newblock {\em Journal of Physics A: Mathematical and Theoretical},
  43(50):505202.

\bibitem[Zusmanovich, 2014]{zusmanovich2014compendium}
Zusmanovich, P. (2014).
\newblock A compendium of {L}ie structures on tensor products.
\newblock {\em Journal of Mathematical Sciences}, 199(3):266--288.

\end{thebibliography}

\end{document}